\documentclass[a4paper,11pt,reqno]{amsart}
\usepackage{amsthm}
\usepackage{amsmath}
\usepackage{enumerate}
\usepackage{amsfonts}
\usepackage{amssymb}
\usepackage{fullpage}
\usepackage{amsmath,amscd}
\usepackage{stmaryrd}
\usepackage{graphicx}
\usepackage{tikz-cd}
\usepackage{array}
\usepackage{amsmath,amssymb,amsfonts,dsfont}
\usepackage[utf8]{inputenc}
\usepackage[english]{babel}
\usepackage[T1]{fontenc}
\usepackage{graphicx}
\usepackage{float}
\usepackage{pdfpages}
\usepackage[a4paper, margin = 3cm, bottom = 3cm]{geometry}
\usepackage{ifpdf}
\usepackage{marginnote}
\usepackage{hyperref}	
\usepackage{caption}
\usepackage{multirow}

\usetikzlibrary{calc}
\usetikzlibrary{decorations.pathreplacing,decorations.markings,decorations.pathmorphing}
\usetikzlibrary{positioning,arrows,patterns}
\usetikzlibrary{cd}
\usetikzlibrary{intersections}
\usetikzlibrary{arrows}

\hypersetup{pdfborder=0 0 0,
	    colorlinks=true,
	    citecolor=black,
	    linkcolor=blue,
	    urlcolor=red,
	    pdfauthor={Guillaume Tahar}
	   }
\newtheorem{thm}{Theorem}[section]
\newtheorem{cor}[thm]{Corollary}
\newtheorem{prop}[thm]{Proposition}
\newtheorem{lem}[thm]{Lemma}

\theoremstyle{definition}
\newtheorem{defn}[thm]{Definition}
\theoremstyle{remark}
\newtheorem{rem}[thm]{Remark}
\theoremstyle{definition}

\theoremstyle{definition}

\theoremstyle{definition}

\numberwithin{equation}{section}
\title{Simplicial arrangements with few double points}

\author{Dmitri Panov}
\address[Dmitri Panov]{Department of Mathematics, King’s College, London, United Kingdom}
\email{dmitri.panov@kcl.ac.uk}

\author[G.~Tahar]{Guillaume Tahar}
\address[Guillaume Tahar]{Beijing Institute of Mathematical Sciences and Applications, Huairou District, Beijing, China}
\email{guillaume.tahar@bimsa.cn}

\date{\today}

\begin{document}

\begin{abstract}
In their solution to the orchard-planting problem, Green and Tao established a structure theorem which proves that in a line arrangement in the real projective plane with few double points, most lines are tangent to the dual curve of a cubic curve. We provide geometric arguments to prove that in the case of a simplicial arrangement, the aforementioned cubic curve cannot be irreducible. It follows that Gr\"{u}nbaum's conjectural asymptotic classification of simplicial arrangements holds under the additional hypothesis of a linear bound on the number of double points.
\end{abstract}
\maketitle
\setcounter{tocdepth}{1}
\begin{center}
\textit{\`{A} la mémoire de Gérard Mahec}
\end{center}

\tableofcontents

\section{Introduction}

A simplicial line arrangement is a finite set $\mathcal{A}$ of lines in the real projective plane $\mathbb{RP}^{2}$ such that every connected component of $\mathbb{RP}^{2} \setminus \bigcup\limits_{L \in \mathcal{A}} L$ is a triangle.
\par
There are two flavours of classification of line arrangements:
\begin{itemize}
    \item Two line arrangements $\mathcal{A}$ and $\mathcal{A}'$ are \textit{combinatorially equivalent} if they induce the same incidence structure on $\mathbb{RP}^{2}$ in terms of vertices, edges and faces.
    \item Two line arrangements $\mathcal{A}$ and $\mathcal{A}'$ are \textit{projectively equivalent} if they are conjugated by a projective transformation of $\mathbb{RP}^{2}$.
\end{itemize}

A combinatorial classification of simplicial line arrangements has been conjectured by Gr\"{u}nbaum in \cite{Gr} but still waits to be proven.
\par
Gr\"{u}nbaum's catalogue of combinatorial equivalence classes contains three infinite families $\mathcal{R}(0)$, $\mathcal{R}(1)$ and $\mathcal{R}(2)$, see below. It also contains about one hundred sporadic arrangements (the most complicated of them is formed by $37$ lines). A comprehensible description of the sporadic arrangements is given in \cite{Gr1}. Thanks to the work of Cuntz, we now know that this list is not entirely complete. A computer verification of simplicial arrangements with at most 27 lines revealed four previously unknown examples (see \cite{Cu1}). More recently, a new sporadic arrangement consisting of 35 lines was discovered (see \cite{Cu2}).

\begin{figure}
\includegraphics[scale=0.8]{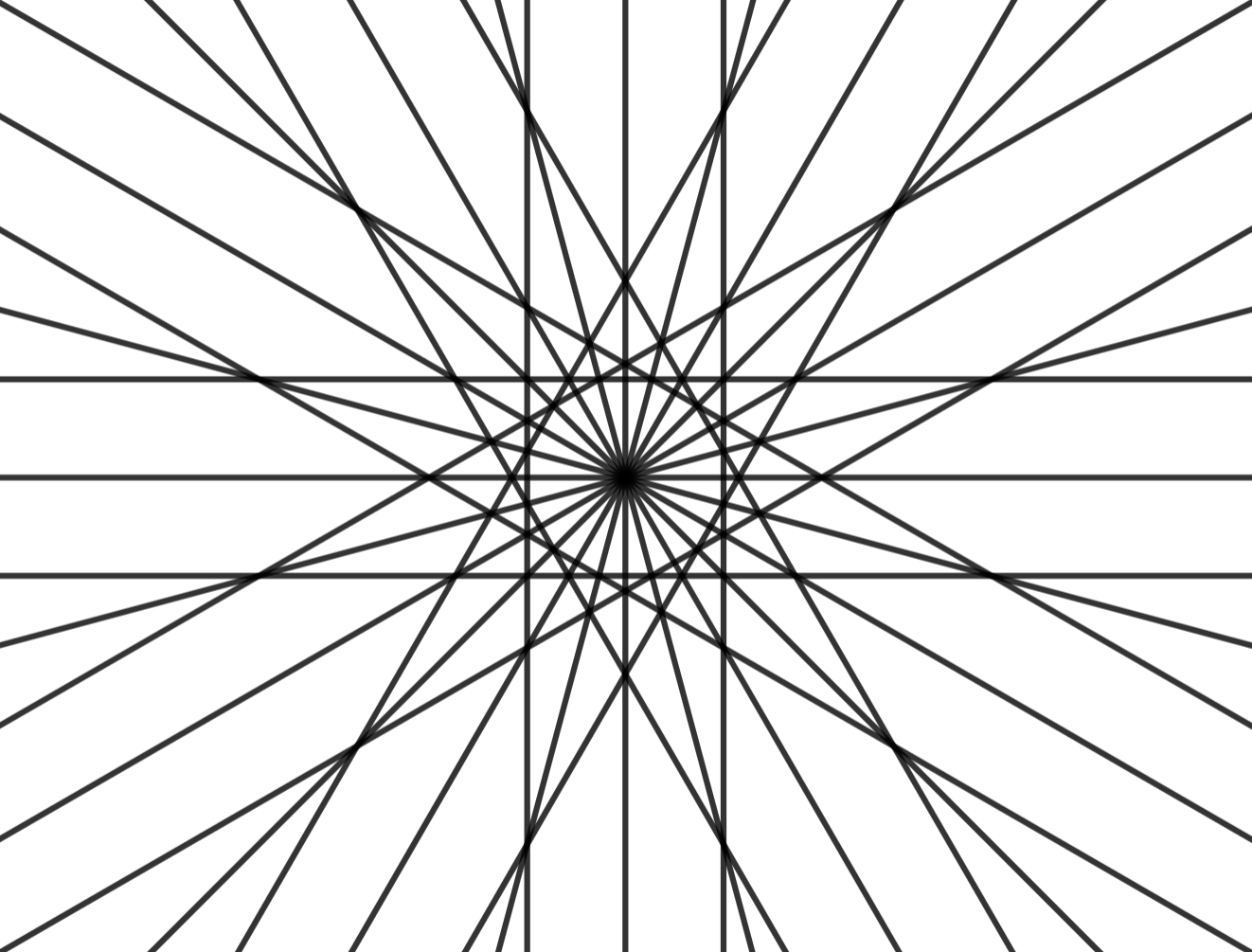}
\caption{Regular simplicial arrangement obtained from a regular $12$-gon.}
\label{fig:regular}
\end{figure}

\begin{defn}
The family $\mathcal{R}(0)$ consists of combinatorial equivalence classes $\mathcal{A}(k+1,0)$ for any $k \geq 2$. Arrangements in $\mathcal{A}(k+1,0)$ are formed by the union of a pencil of $k$ lines passing through a common base point $C$ and an additional line that does not pass through $C$.
\end{defn}

\begin{defn}
The family $\mathcal{R}(1)$ consists of combinatorial equivalence classes $\mathcal{A}(2k,1)$ for $k \geq 3$. Arrangements of $\mathcal{A}(2k,1)$ are combinatorially equivalent to the \textit{regular simplicial arrangement} formed by the $k$ sides and the $k$ symmetry axes of a regular $k$-gon (see Figure~\ref{fig:regular}).
\end{defn}

\begin{defn}
The family $\mathcal{R}(2)$ consists of combinatorial equivalence classes $\mathcal{A}(2k+1,1)$ for $k \in 2\mathbb{N}+4$. Arrangements of $\mathcal{A}(2k+1,1)$ are combinatorially equivalent to the arrangement obtained by adding the line at infinity to the regular simplicial arrangement associated to the regular $k$-gon.
\end{defn}

\begin{rem} Note that arrangements of families  $\mathcal{R}(1)$ and $\mathcal{R}(2)$ are projectively rigid. In other words, two combinatorially equivalent arrangement are identical up to a projective transformation. This has been proved in Theorem~3.6 and Corollary~3.7 of \cite{Cu}. For completeness we prove this fact in the Appendix in Corollary~\ref{cor:ProRig}, deducing it from an elementary and more general Proposition~\ref{prop:chareg}.
\end{rem}

The main result of this paper is a proof of the conjectural asymptotic classification of Gr\"{u}nbaum under the additional hypothesis that the simplicial arrangement contains few double points (vertices where exactly two lines of the arrangement intersect).

\begin{thm}\label{thm:MAIN}
For any constant $K > 0$, there are only finitely many sporadic simplicial arrangements (combinatorial classes not belonging to $\mathcal{R}(0)$, $\mathcal{R}(1)$ or $\mathcal{R}(2)$) that have at most $Kn$ double points, where $n$ is the number of lines of the arrangement.
\end{thm}

%\begin{thm}\label{thm:MAIN}[STATEMENT 2]
%Let $\mathcal{A}$ be a simplicial arrangement of $n$ lines in the projective plane $\mathbb{RP}^{2}$. Let $K > 0$ be a real parameter. Suppose that at most $Kn$ vertices of $\mathcal{A}$ are double points and that $n \geq  exp~exp (cK^{c})$ for some constant $c$. Then arrangement $\mathcal{A}$ belongs to one of the three infinite families $\mathcal{R}(0)$, $\mathcal{R}(1)$ or $\mathcal{R}(2)$.
%\end{thm}

The proof of Theorem~\ref{thm:MAIN}, given in Section~\ref{sub:PROOF}, relies on several results that relate the combinatorial structure of line arrangements to the geometry of cubic curves. The most important of them is the following theorem of Green-Tao\footnote{In \cite{GT}, the structure theorem is stated in the projectively dual terms of a configuration of points with few ordinary lines (lines containing exactly two points of the configuration). In particular, the regular simplicial arrangement appears as its projective dual 
$$X_{m} = \lbrace{ [\cos \frac{2\pi j}{m} : \sin \frac{2\pi j}{m} : 1], 1 \leq j \leq m \rbrace} \cup 
\lbrace{ [-\sin \frac{2\pi j}{m} : \cos \frac{2\pi j}{m} : 0], 1 \leq j \leq m \rbrace}
$$ for some $m$.}, proved as Theorem~1.5 in \cite{GT}.

\begin{thm}[Green-Tao full structure theorem]\label{thm:GT}
Suppose that $\mathcal{A}$ is an arrangement of $n$ lines in $\mathbb{RP}^{2}$. Let $K > 0$ be a real parameter. Suppose that at most $Kn$ vertices of $\mathcal{A}$ are double points and that $n \geq  exp~exp (cK^{c})$ for some constant $c$ independent from $K$. Then, up to a projective transformation, $\mathcal{A}$ differs by at most $O(K)$ lines (which can be added or deleted) from an example of one of the following three types:
\begin{enumerate}
    \item a pencil of $n$ lines passing through a same point;
    \item a regular simplicial arrangement formed by the sides and symmetry axes of a regular $m$-gon where $m=\frac{n}{2}+O(K)$;
    \item the projective dual of a coset $H \oplus g$, $3g \in H$ of a finite subgroup $H$ of the non-singular real points of an irreducible cubic curve with $H$ having cardinality $n+O(K)$.
\end{enumerate}
\end{thm}
This theorem says that arrangements with a linear bound on the number of double points are \emph{governed by cubic curves}. Namely we see that in all three cases (1), (2), (3) when we take the arrangement of points dual to the lines of $\mathcal A$, most of such points lie on some cubic curve. The case (1) is a bit degenerate, and dual points lie on a line, in case (2) they lie on a disjoint union of a conic and a line, and in case (3) on an irreducible cubic curve. As for the arrangement $\mathcal A$ itself, most of its lines are \emph{tangent} to the curve  projectively dual to a cubic curve (again cases (1) and (2) are a bit degenerate here).

%This theorems says in particular that for any arrangement $\mathcal A$ of $n$ lines with at most $Kn$ double points there exists a \emph{dual cubic} in $\mathbb {RP}^2$ such that all except $O(K)$ lines of $\mathcal A$ are tangent to it. Case (1) is a bit degenerate, and the primary cubic can be chosen as three generic lines. In case (2) the primary cubic is the union of  

%A geometric interpretation of the few double points hypothesis is as follows. For each simplicial arrangement, we can construct a polyhedral metric on $\mathbb{RP}^{2}$ by gluing equilateral triangles according to the incidence structure of the arrangement. Double points will be vertices of positive curvature. Since the total curvature is a topological invariant (Gauss-Bonnet equality), the number of double points controls the number of vertices of order at least four (vertices of negative curvature).
\par
%If there is a small number of double points, then most vertices in the arrangements are triple points and large domains of the arrangement will be covered by an equilateral lattice. This argument has been developed by Melchior in his proof of Sylvester-Gallai theorem, see \cite{M}.\newline

Case (1) of Theorem~\ref{thm:GT} is realized by the arrangements of family $\mathcal{R}(0)$ while case (2) is realized by the regular simplicial arrangements of families $\mathcal{R}(1)$ and $\mathcal{R}(2)$. However in Section~\ref{sub:Reduction}, we prove that no curve that is projectively dual to an irreducible cubic can be tangent to most lines of a simplicial arrangement, eliminating case (3) of Theorem~\ref{thm:GT} for simplicial arrangements.

\begin{prop}\label{prop:irreducible}
Let $\mathcal{A}$ be a simplicial arrangement of $n$ lines in $\mathbb{RP}^{2}$. If the dual points of $k$ lines of $\mathcal{A}$ belong the same irreducible cubic curve then $n$ and $k$ satisfy $n \geq \frac{8}{7}k-3$.
\end{prop}

Finally, we prove in Section~\ref{sec:rigidity} that a simplicial arrangement cannot differ from a regular arrangement by only few lines unless it is actually a regular arrangement (up to the addition of the line at infinity in the case of arrangements obtained from a regular polygon with an even number of sides). This is the  final step in the proof of Theorem~\ref{thm:MAIN}.\newline

\emph {Acknowledgements.} We are grateful to Sergey Galkin and Alexey Panov for discussions of dual cubics. The second author is sincerely grateful to the King's College for the hospitality in June 2022.\newline

\emph {Data availability statement.} Data sharing not applicable to this article as no datasets were generated or analysed during the current study.

\section{Generalities about simplicial arrangements}\label{sec:generalities}

\subsection{Conventions}

An arrangement of finitely many line defines an embedded graph in the real projective plane $\mathbb{RP}^{2}$. {\it Vertices} and {\it edges} of the arrangement are the vertices and edges of the graph. {\it Faces} are the connected components of the complement of the graph in $\mathbb{RP}^{2}$, these are projective polygons cut out by lines. A line arrangement is \textit{simplicial} if each face is a triangle.\newline

We use the term \textit{incident} in the standard way to describe incidence relations between vertices, edges and triangles of a simplicial arrangement. We also use the term incident for a vertex belonging to a line.
\par
We use the term \textit{adjacent} as follows:
\begin{itemize}
\item two vertices are \textit{adjacent} if they are the endpoints of a same edge of the arrangement;
\item one vertex and one edge are \textit{adjacent} if they belong to the boundary of a same triangle of the arrangement.
\end{itemize}

For two vertices $X,Y$, we denote by $XY$ the unique line incident to both $X$ and $Y$. We may denote by $[XY]$ one of the two segments between $X$ and $Y$ in line $XY$.

\begin{defn}
In a line arrangement, the \textit{order} $k_{X}$ of a vertex $X$ is the number of its incident lines.
\end{defn}

%The next result follows from Gauss-Bonnet formula applied to the polyhedral metric on $\mathbb{RP}^{2}$
%obtained by replacing every face of a simplicial arrangement by an equilateral triangle.

%\begin{lem}\label{lem:Euler}
%For any simplicial arrangement, we have $\sum \limits_{k \geq 2} v_{k}(k-3) =-3$ where $v_{k}$ is the number of vertices incident to $k$ lines of the arrangement.
%\end{lem}

%Lemma \ref{lem:Euler} implies in particular that every simplicial arrangement contains at least three double points. It also motivates the convention according to which \textit{vertices of higher order} are the vertices of order at least four.

\subsection{Exceptional properties of family $\mathcal{R}(0)$}

Near-pencils, or equivalently simplicial arrangements of combinatorial classes $\mathcal{A}(k+1,0)$ for $k \geq 2$ are very special among simplicial arrangements. The following result appears as Lemma 3.2 in \cite{Cu1} in the more general context of simplicial wirings, where it is presented without proof as it is considered folklore. We include a proof here for the sake of completeness.

\begin{lem}\label{lem:adjaDP}
A simplicial arrangement contains two adjacent double points if and only if it belongs to family $\mathcal{R}(0)$.
\end{lem}

\begin{proof}
If an edge is incident to two double points, then the union of the two triangles incident to the edge is a digon (i.e. a polygon with two edges) whose extremal points coincide in $\mathbb{RP}^{2}$. Therefore, every other line of the arrangement intersects the digon at its unique vertex. This describes exactly a near-pencil or equivalently an arrangement of $\mathcal{R}(0)$.
\end{proof}

\begin{lem}\label{lem:R0edge} Suppose $A$ and $B$ are two adjacent vertices in a simplicial arrangement $\mathcal{A}$. Then at least one the following claims holds:
%In a simplicial arrangement $\mathcal{A}$, if two vertices $A,B$ are adjacent, then at least one the following propositions holds:
\begin{itemize}
\item there is exactly one edge $[AB]$ such that $A$ and $B$ both are incident to $[AB]$;
\item $\mathcal{A}$ belongs to family $\mathcal{R}(0)$.
\end{itemize}
\end{lem}

\begin{proof}
If two vertices $A$ and $B$ of a line arrangement are the endpoints of two distinct edges $E$ and $F$, then $E \cup F$ form a line of the arrangement (because exactly one line of the projective plane passes through two given points). In particular, the are exactly two edges incident to both $A$ and $B$.
\par
Since every other line of the arrangement crosses $E \cup F$,  every line passes through either $A$ or $B$. Therefore, any vertex distinct from $A$ and $B$ is a double point where a line incident to $A$ intersects a line incident to $B$.
\par
Assuming $A$ is a vertex of order $k$, $2k$ edges of arrangement $\mathcal{A}$ are incident to $A$. Two of them are incident to $B$ while the $2k-2$ other edges are incident to double points. Unless $A$ is a double point, two double points incident to $A$ are adjacent. Since the arrangement is simplicial, Lemma~\ref{lem:adjaDP} implies that the arrangement belongs to family $\mathcal{R}(0)$. The same reasoning applies to $B$. If every vertex of $\mathcal{A}$ is a double point, then Lemma~\ref{lem:adjaDP} again proves that the arrangement belongs to $\mathcal{R}(0)$. In other words, it is a near-pencil where every line except one is incident to a point (which is a double point). In this case, we obtain arrangement $\mathcal{A}(3,0)$.
\end{proof}

\subsection{Star of a vertex}\label{sub:star}

\begin{defn}
In a simplicial arrangement, the \textit{star} $S(A)$ of a vertex $A$ is the union of all triangles incident to $A$.
\par
The \textit{exterior vertices} of the star $S(A)$ are the vertices of $S(A)$ that are distinct from $A$.
\par
The \textit{exterior edges} of $S(A)$ are sides of triangles of $S(A)$ that are opposite to $A$.
\end{defn}

We deduce strong results on the star of a vertex in simplicial arrangements that do not belong to $\mathcal{R}(0)$.

\begin{prop}\label{prop:star1}
Let $\mathcal{A}$ be a simplicial arrangement that does not belong to family $\mathcal{R}(0)$ and let $C$ be a vertex of order $k$. Then the following statements hold:
\begin{itemize}
    \item The star $S(C)$ of $C$ is a convex polygon.
    \item Two exterior edges of $S(C)$ belong to the same line if and only if they are consecutive and their common  vertex is a double point.
    \item The boundary of $S(C)$ has $2k$ distinct exterior vertices. Among them, there are $d \leq k$ double points and the exterior edges of $S(C)$ belong to $2k-d \geq k$ distinct lines, all of which are not incident to $A$.
\end{itemize}
\end{prop}

\begin{proof}
Consider an exterior edge $[AB]$ of $S(C)$. Then line $AB$ cannot contains $C$ because $A,B,C$ are the vertices of a triangle. We define $\mathcal{F}_{C}$ as the subarrangement of $\mathcal{A}$ formed by lines containing an exterior edge of $S(C)$. These lines are not incident to $C$ so they do not intersect the interior of $S(C)$. This implies that $S(C)$ is a face of subarrangement $\mathcal{F}_{C}$. For an arrangement that does not belong to family $\mathcal{R}(0)$, vertices of the boundary of the star are distinct from each other. Otherwise, one exterior vertex would be connected to the central one with two distinct edges, see Lemma~\ref{lem:R0edge}. In particular, $S(C)$ cannot be a digon, $\mathcal{F}_{C}$ contains at least three lines and $S(C)$ is a convex polygon of $\mathbb{RP}^{2}$.
\par
Consider a line $L$ that contains an exterior edge of $S(C)$. $L \cap S(C)$ is convex so it should be entirely contained in the boundary of $S(C)$. This implies that $L \cap S(C)$ is the union of consecutive exterior edges of $S(C)$. We assume $[AB]$ and $[BD]$ are two such consecutive exterior edges of $S(C)$ that belong to the same line $L=BD$. Thus $B$ is a double point (the other line incident to $B$ being $BC$).
\par
Since there is no pair of adjacent double points in an arrangement that does not belong to $\mathcal{R}(0)$ (see Lemma \ref{lem:adjaDP}), every line contains at most two exterior edges of $S(C)$. The number $d$ of double points among exterior vertices of $S(C)$ thus satisfies $d \leq k$. Since a line contains one or two exterior edges of $S(C)$ depending whether the common exterior vertex is a double point or not, subarrangement $\mathcal{F}_{C}$ contains at least $2k-d$ distinct lines.
\end{proof}

Each arrangement cuts $\mathbb {RP}^2$ into several convex polygons. It turns out that in case the number of sides of one of these polygons coincides with the number of lines of the arrangement, the arrangement has a particularly simple form.

\begin{lem}\label{lem:centralpolygon}
If one face of an arrangement $\mathcal{A}$ of $m \geq 3$ lines is a convex $m$-gon $\mathcal{P}$ in $\mathbb{RP}^{2}$, then $\mathcal{A}$ is combinatorially equivalent to the arrangement formed by $m$ tangent lines of a unit circle in such a way that the tangency points form a regular $m$-gon. In particular, every vertex of $\mathcal{A}$ is a double point and the faces are polygon $\mathcal{P}$, $m$ triangles adjacent to $\mathcal{P}$ and $\frac{m(m-3)}{2}$ quadrilaterals.
\end{lem}

\begin{proof}
Since $\mathcal{P}$ is a convex domain, for every point $x$ outside $\mathcal{P}$, there are two lines $L$ and $L'$ incident to $x$ that are {\it tangent} to $\mathcal{P}$, i.e. they intersect $\mathcal{P}$ but do not cross its interior. These two lines define four angular sectors at $x$. For two of them, lines incident to $x$ belonging to these sectors do not intersect $\mathcal{P}$. For the two others, they cross the interior of $\mathcal{P}$ (and cut its boundary twice).
\par
Consequently, there is no point on $\mathbb{RP}^{2}$ that is incident to more than two lines tangent to $\mathcal{P}$. This implies every vertex of $\mathcal{A}$ is a double point. Every such arrangement is characterized by the shape of its central polygon $\mathcal{P}$. The moduli space of convex $m$-gons in $\mathbb{RP}^{2}$ is path connected and along every path in this moduli space, no vertex of higher order appears. This implies that the incidence structures of each of these arrangements are isomorphic to the incidence structure of the arrangement formed by $m$ regularly disposed tangent lines along a circle.
\end{proof}

\section{Recollection on cubic curves and their duals}\label{sub:classification}
In this short section, we give a qualitative description of real cubic curves and their duals. We denote by $\hat\Gamma$ a cubic curve in $\mathbb{RP}^{2}$ and by $\Gamma$ the projectively dual curve. By definition, $\Gamma$ is formed by all lines tangent to $\hat\Gamma$, where for a singular point  $x\in\hat\Gamma$ all lines containing $x$ are tangent to $\hat\Gamma$. Taking multiplicities into account, $\Gamma$ is defined by a degree $6$ polynomial. 

Recall that every irreducible planar cubic curve $\hat \Gamma$ can be written in Weierstrass form $y^2=x^3+ax+b$. An irreducible cubic curve 
can have at most one singularity, namely, a double point, a cusp or an isolated point (see Theorem~8.4 in \cite{Bix}), this happens when the polynomial $x^3+ax+b$ has a double root. Altogether, we have $3$ such irreducible cubic curves up to an isomorphism:
\begin{itemize}
    \item $\hat{\Gamma}$ is a \textit{nodal irreducible cubic} ($y^2=(x-1)^2(x+2)$):
    %$y^{2}=(x-0.5)^2(x+1)=x^3-0.75x+0.25$
    the dual cubic $\Gamma$ is a union of a closed curve with one cusp and its unique bitangent line.
    \item $\hat{\Gamma}$ is a \textit{semicubical parabola} ($y^2=x^3$): the dual cubic $\Gamma$ is the union of a semicubic parabola and the line at infinity.
    \item $\hat{\Gamma}$ is an \textit{acnodal irreducible cubic} ($y^{2}=(x+1)^2(x-2)$): the dual cubic $\Gamma$ is the disjoint union of a curvilinear triangle with three cusps and a line (dual to the isolated singular point of $\hat{\Gamma}$).
\end{itemize}

Next, a smooth irreducible cubic can consist of either one or two connected components. One component intersects every line in $\mathbb{RP}^2$ and has three inflection points, while the optional component is a convex oval:
\begin{itemize}
    \item $\hat{\Gamma}$ is a smooth irreducible cubic with one connected component: 
    %$y^2=x^3-0.8x+0.7$ 
    the dual cubic $\Gamma$ is an embedded\footnote{Indeed $\Gamma$ has no self-intersections, since $\hat\Gamma$ has no bitangents} curvilinear triangle with three cusps joined pairwise by three locally convex arcs.
    \item $\hat{\Gamma}$ is a smooth irreducible cubic with two connected components: 
    %$y^2=x^3-1.245x+0.505$
    the dual cubic $\Gamma$ is a curvilinear triangle with three cusps (as above) contained in an oval.
\end{itemize}

\begin{table}[ht]
\centering
\begin{tabular}{|c|c|c|c|c|} % l = row label, cccc = 4 image columns
  \hline
  
    & \textbf{Nodal} & \textbf{Acnodal} & \textbf{Smooth} & \textbf{Smooth} \\
    & & & \textbf{connected} & \textbf{disconnected}\\
    %\multicolumn{5}{c}{\rule{17.5cm}{1pt}}\\
    %\noalign{\hrule height 1.2pt}
    \hline
    \noalign{\vskip 0.3mm}
    \raisebox{1.7cm}{\textbf{Primary}}&
    %\textbf{Primary} &
   
    \includegraphics[width=0.16\textwidth]{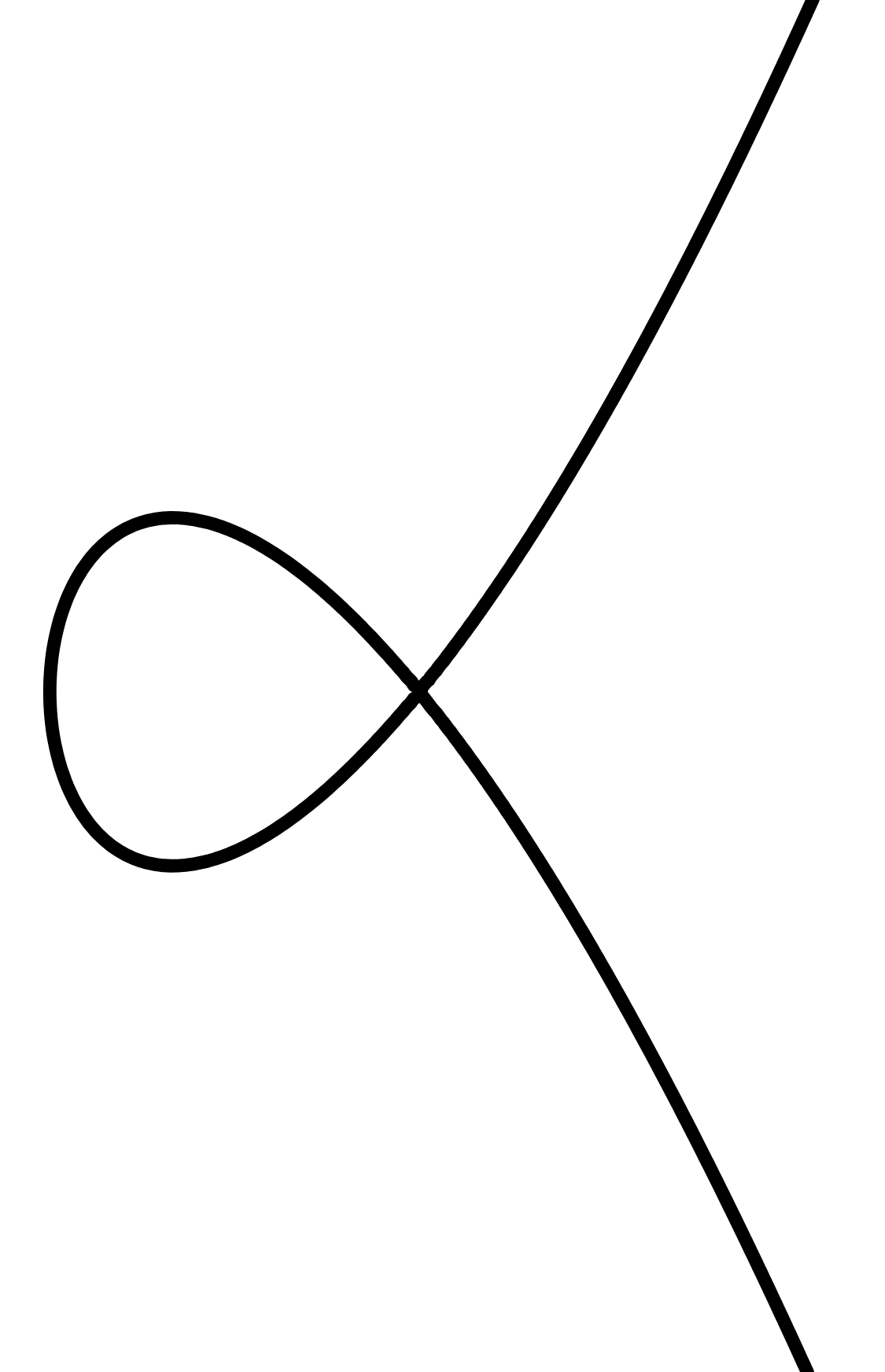} &
    \includegraphics[width=0.17\textwidth]{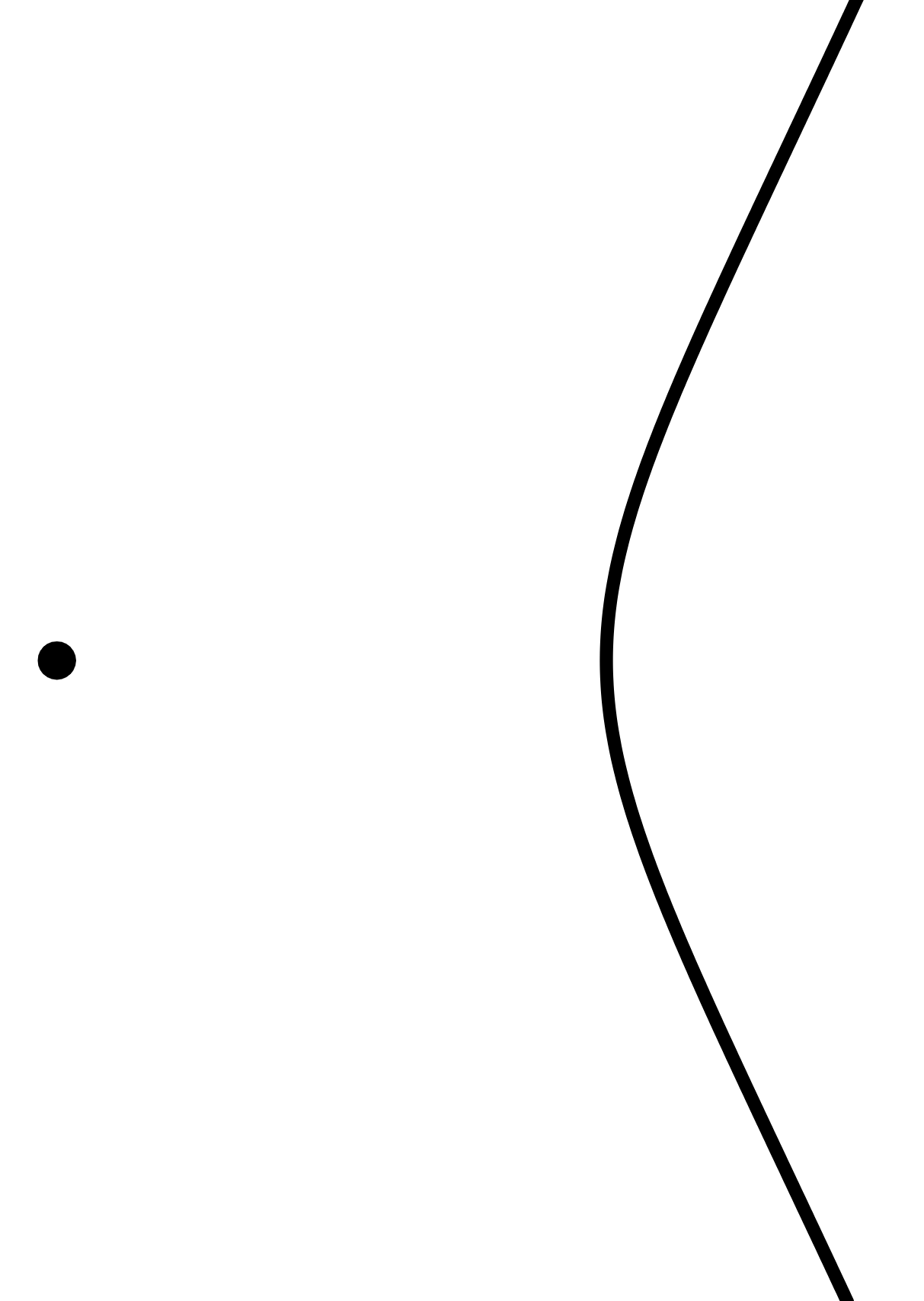} &
    \includegraphics[width=0.18\textwidth]{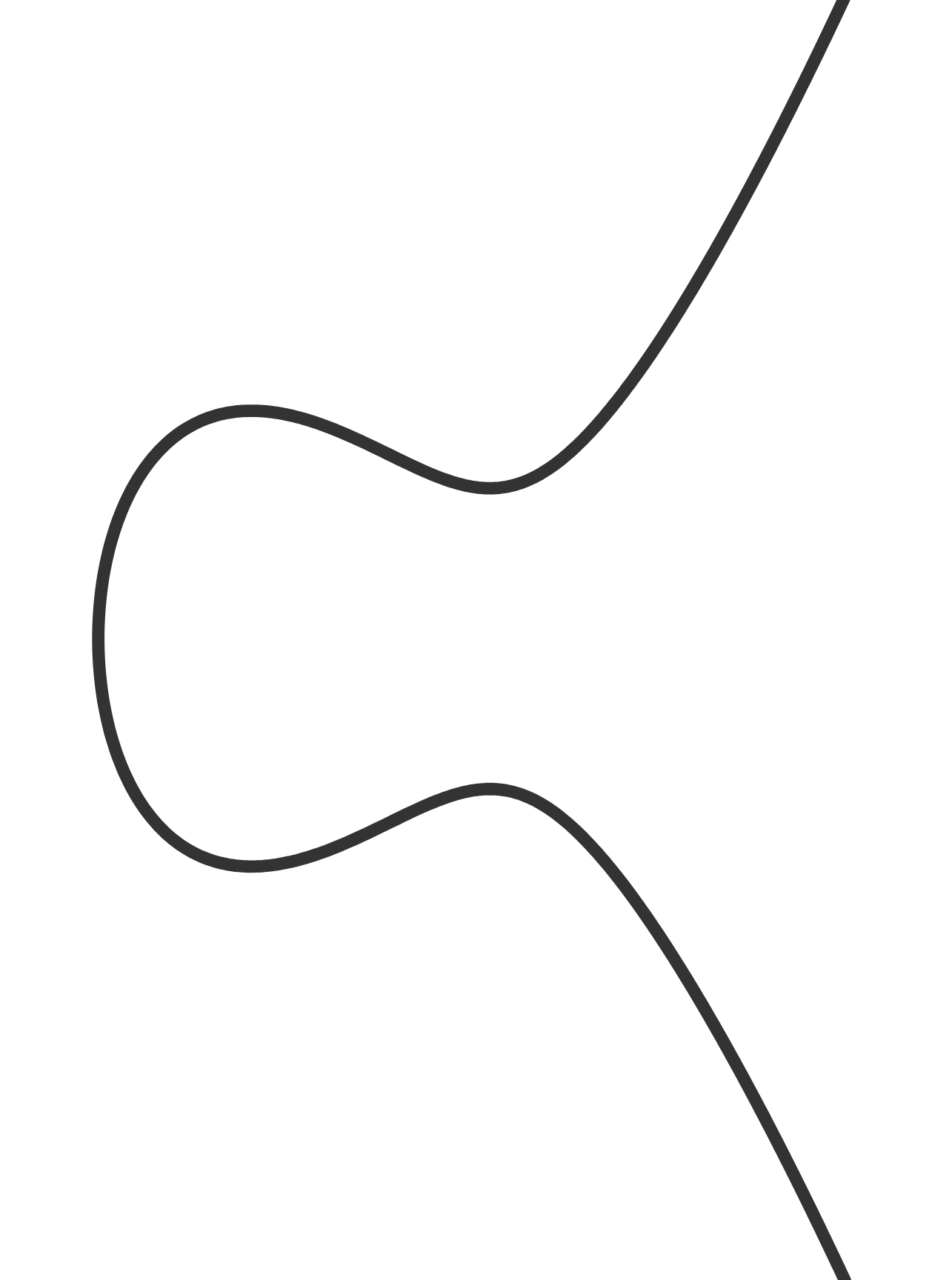} &
    \includegraphics[width=0.20\textwidth]{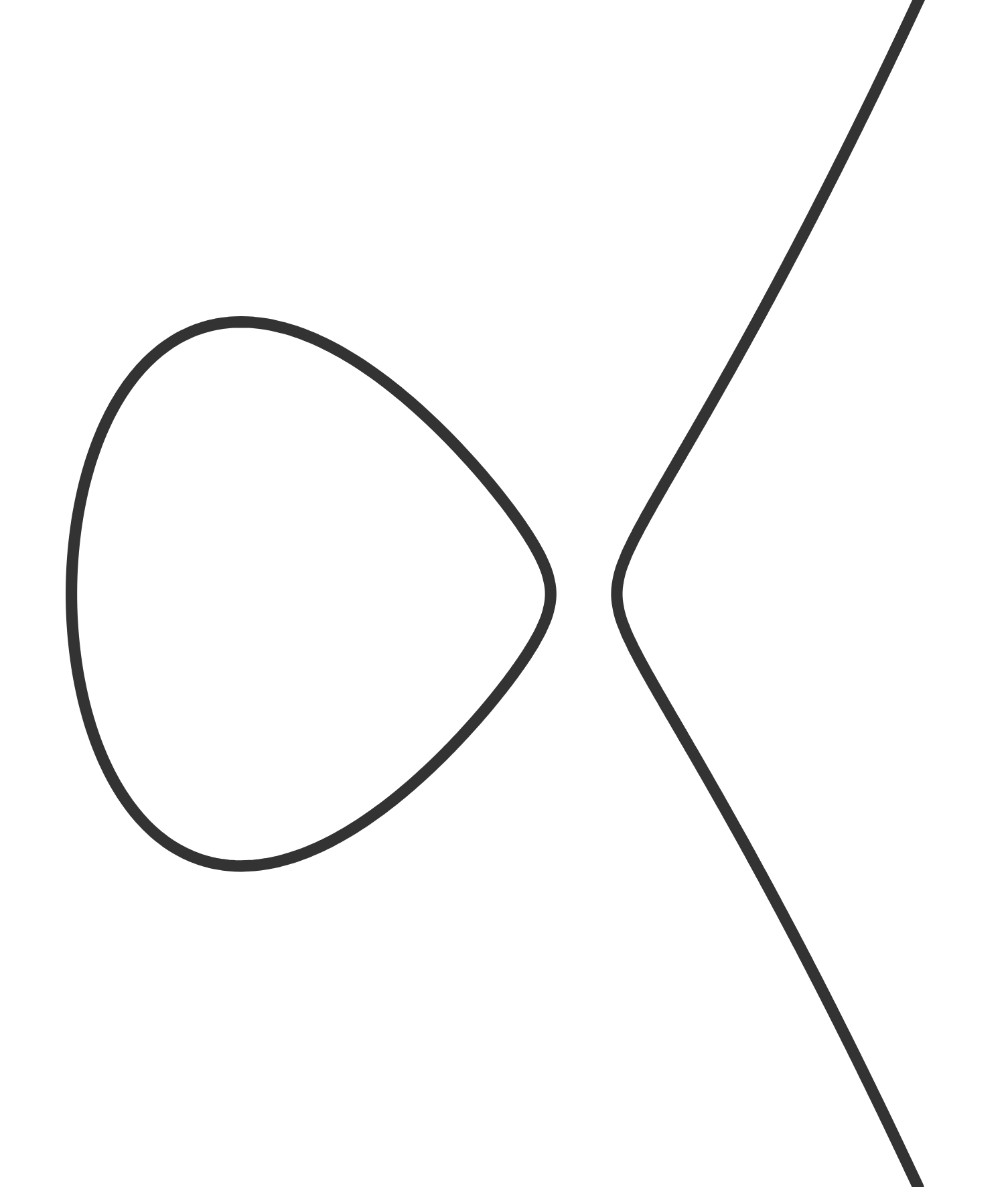} \\
     \hline 
     \noalign{\vskip 0.3mm}
     %\noalign{\hrule height 1.2pt}
    \raisebox{1.5cm}{\textbf{Dual}} &
    \includegraphics[width=0.13\textwidth]{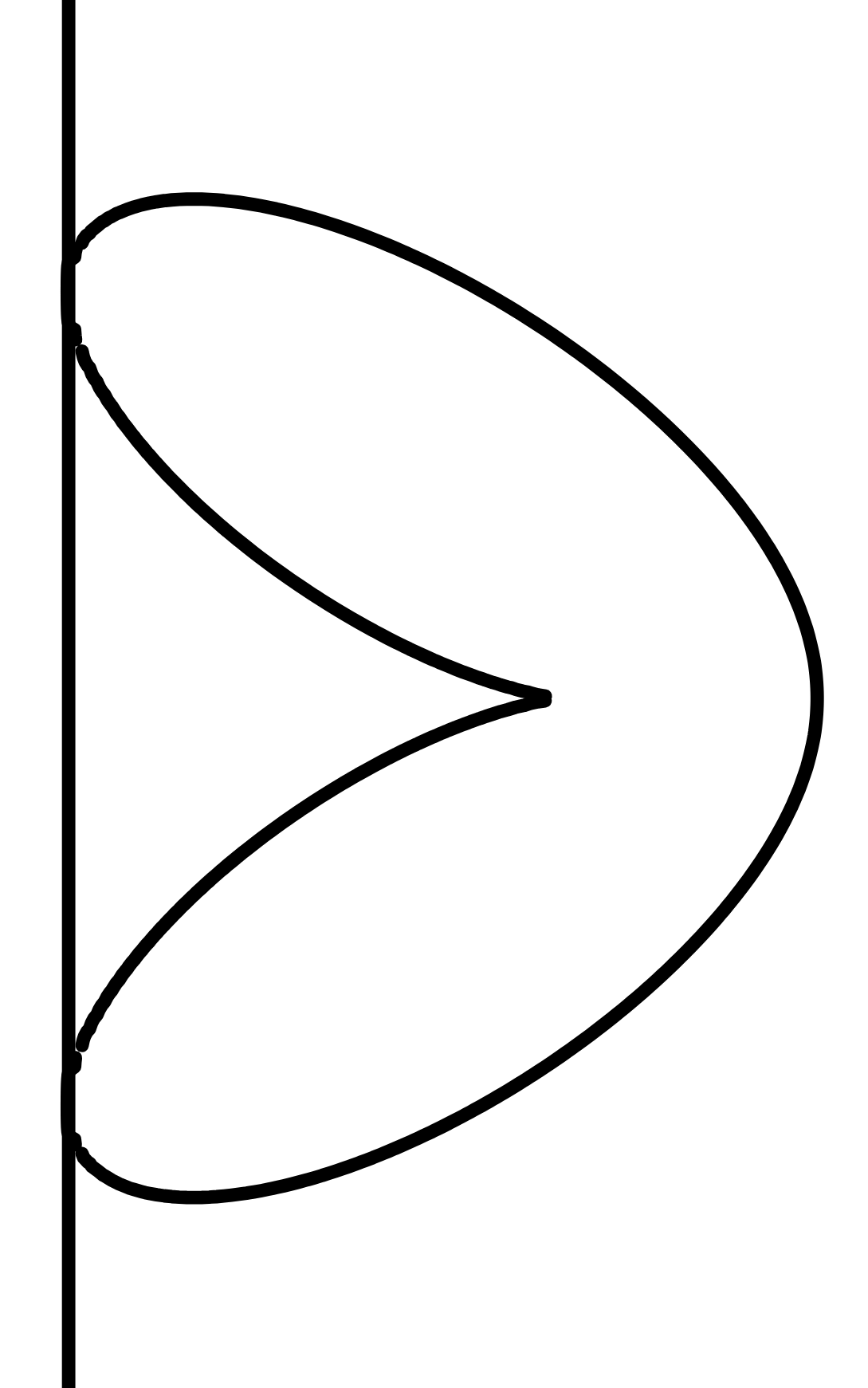} &
    \raisebox{0.25cm}{\includegraphics[width=0.17\textwidth]{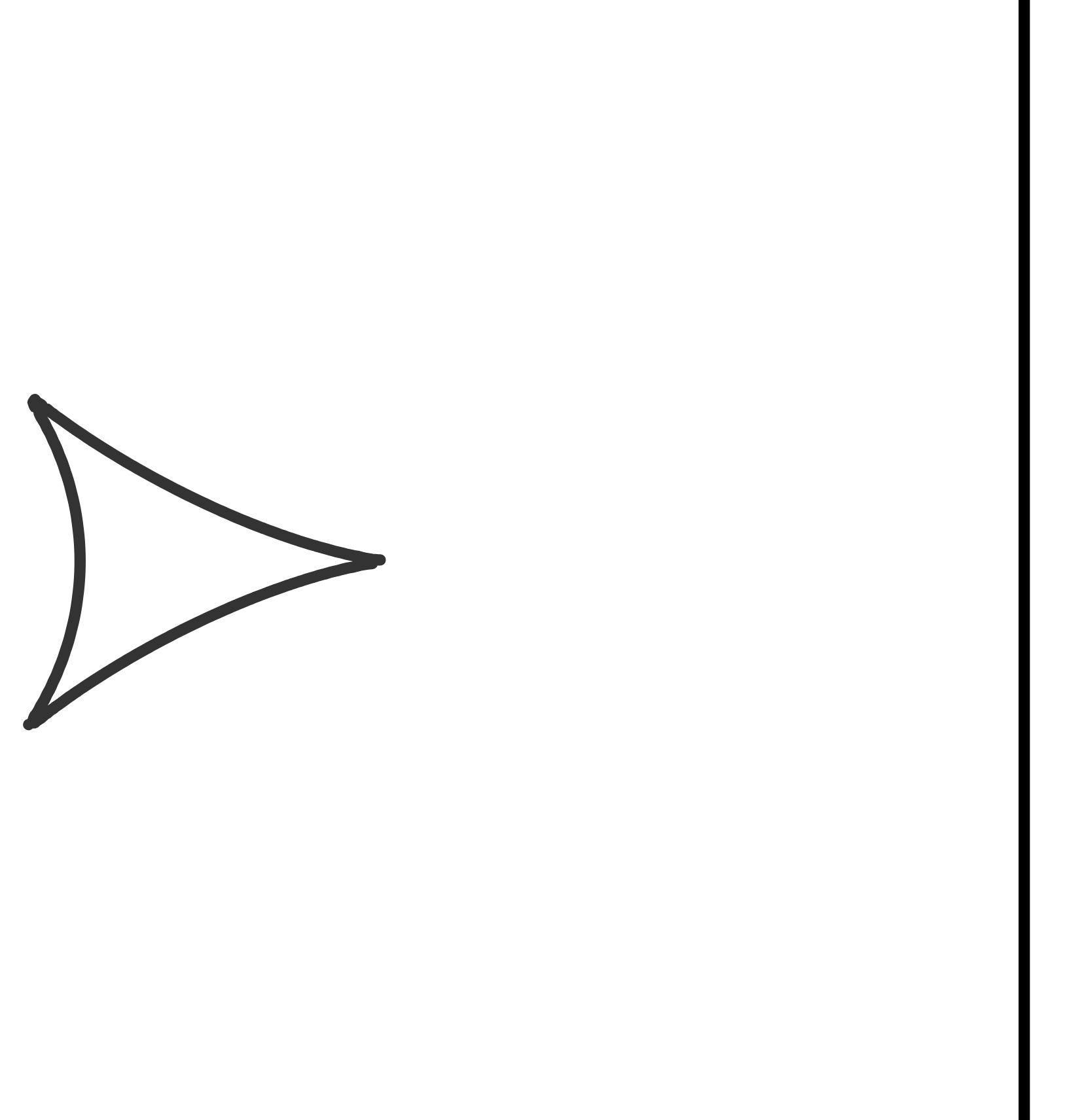}} &
    \hspace{0.8cm}\raisebox{0.25cm}{\includegraphics[width=0.29\textwidth]{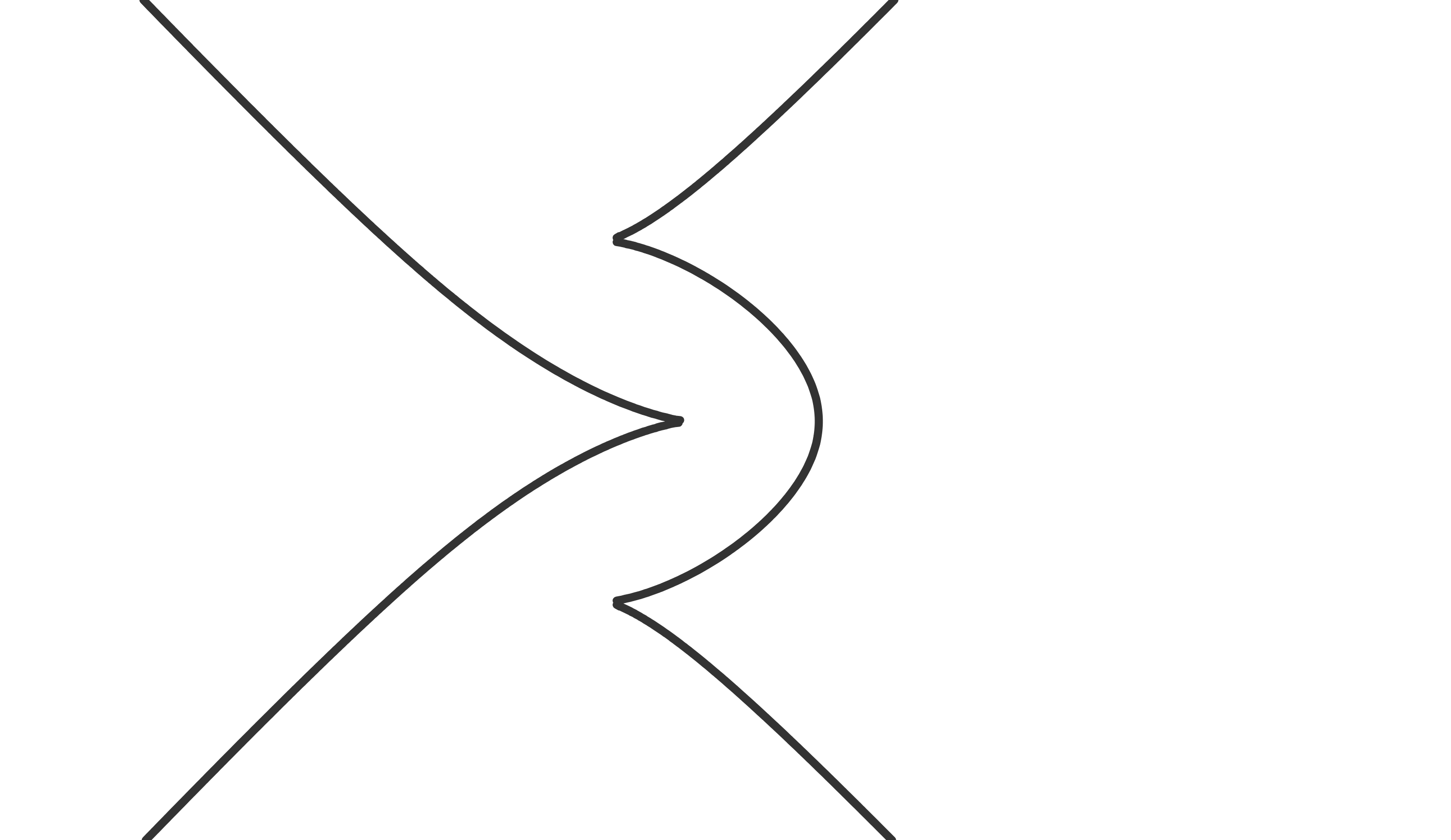}} &
    \includegraphics[width=0.18\textwidth]{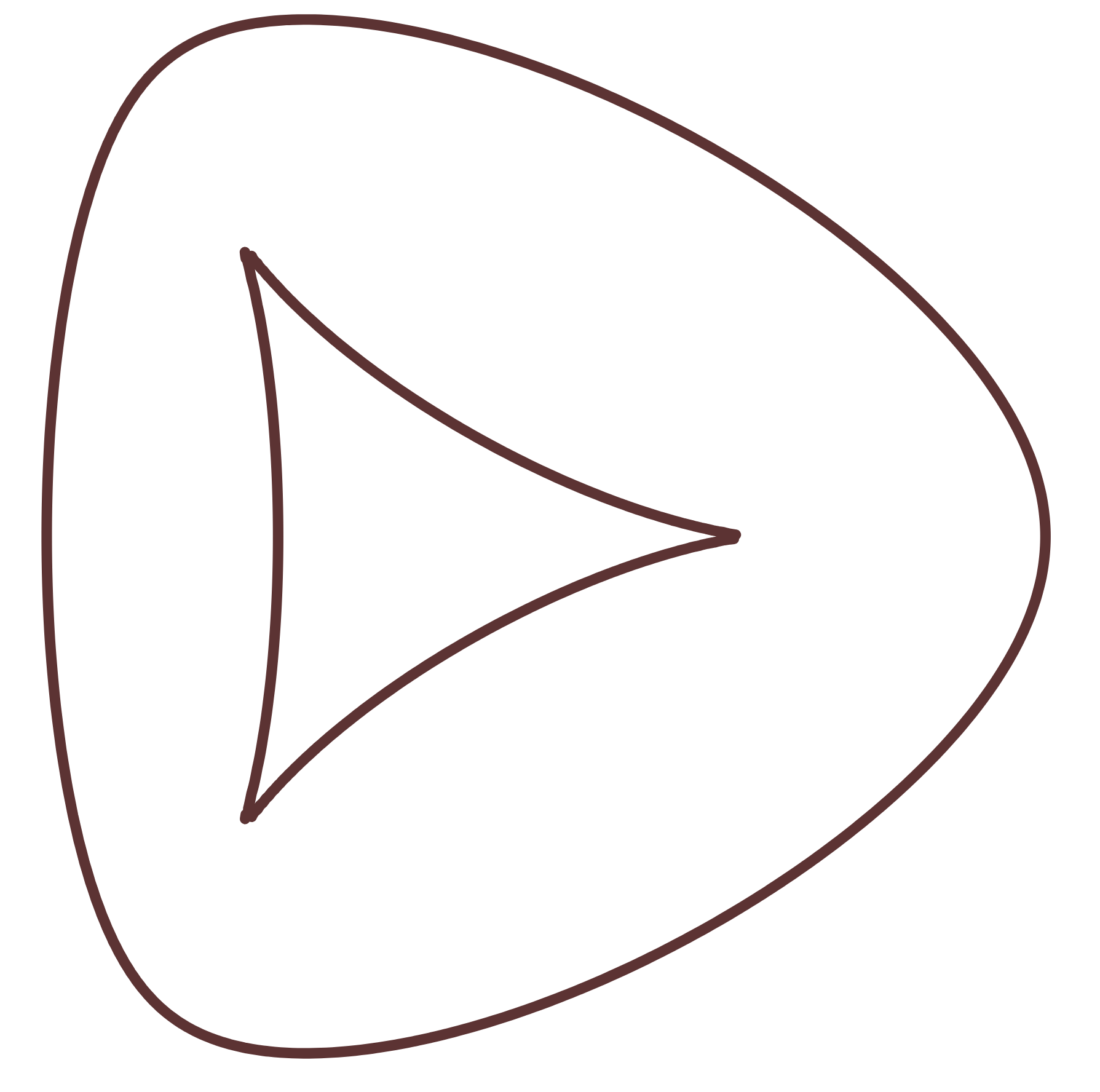} \\
      \hline
      %\noalign{\hrule height 1.2pt}
\end{tabular}
\caption{Some cubic curves and their projective duals}
\label{tab:Cubics}
\end{table}

Examples of irreducible cubic curves and their dual curves are given in Table~\ref{tab:Cubics}. Here the cubic curve $\hat \Gamma$ is called primary, and the dual cubic curve $\Gamma$ is called dual. We don't provide the figure for the easiest case - the semi-cubic parabola.

\section{Estimates on lines tangent to a dual cubic in a simplicial arrangement}\label{sec:estimates}

The main goal of this section is to prove Proposition~\ref{prop:irreducible}.

%\subsection{Description of dual cubics}\label{sub:classification}

\subsection{Families of lines decomposing a convex polygon into triangles}

The following lemma will be the main tool to prove the estimates of Proposition~\ref{prop:irreducible}.

\begin{lem}\label{lem:CPInter}
Let $(\mathcal{F},\mathcal{G})$ be a pair of disjoint families of lines intersecting the interior of a convex polygon $\mathcal{P}$ (with at least three sides) of $\mathbb{RP}^{2}$ such that:
\begin{itemize}
    \item $\mathcal{F}\cup \mathcal{G}$ cuts out $\mathcal{P}$ into triangles;
    \item lines of $\mathcal{F}$ intersect each other outside $\mathcal{P}$ (including its boundary).
\end{itemize}
Then we have $|\mathcal{G}| \geq |\mathcal{F}|-1$.
\end{lem}

\begin{proof}
Assume by contradiction that the statement is false and take  a counter-example $(\mathcal{P},\mathcal{F},\mathcal{G})$  for which  $|\mathcal{G}|$ is minimal possible. We clearly have $|\mathcal{F}| \geq 2$. Since lines $\Lambda_{1},\dots,\Lambda_{|\mathcal{F}|}$ of $\mathcal{F}$ intersect each other outside $\mathcal{P}$, they cut $\mathcal{P}$ into $|\mathcal{F}|+1$ polygons $\mathcal{P}_{1},\dots,\mathcal{P}_{|\mathcal{F}|+1}$ in such a way that for any $1 \leq i \leq |\mathcal{F}|$ line $\Lambda_{i}$ contains a side of both $\mathcal{P}_{i}$ and $\mathcal{P}_{i+1}$.
\par
Assume by contradiction that some line $L$ of $\mathcal{G}$ intersects the interior of $\mathcal{P}_{1}$. Then the interior of one connected component $\mathcal{P}'$ of $\mathcal{P}\setminus L$ intersects every line of $\mathcal{F}$. We denote by $\mathcal{G}'$ the set of lines of $\mathcal{G}$ that intersect the interior of $\mathcal{P}'$. Triple $(\mathcal{P}',\mathcal{F},\mathcal{G}')$ is also a counter-example while $|\mathcal{G}'|\leq |\mathcal{G}|-1$. Then $(\mathcal{P},\mathcal{F},\mathcal{G})$ is not minimal. Consequently $\mathcal{P}_{1}$ is a triangle and no line of $\mathcal{G}$ intersects its interior.
\par
Since $|\mathcal{F}|\geq 2$, $\mathcal{P}_{2}$ is a polygon with at least four sides. Besides, no line of $\mathcal{G}$ intersects the interior of the edge $\mathcal{P}_{1} \cap \mathcal{P}_{2}$. Thus at least one line $L'$ of $\mathcal{G}$ contains an end vertex of $\mathcal{P}_{1} \cap \mathcal{P}_{2}$ (otherwise $\mathcal{P}_{2}$ would have just three sides). Then,
the interior of at least one connected component $\mathcal{P}'$ of $\mathcal{P} \setminus L'$ intersects lines $\Lambda_{2},\dots,\Lambda_{|\mathcal{F}|}$.
\par
We define again $\mathcal{G}'$ as the set of lines of $\mathcal{G}$ that intersect the interior of $\mathcal{P}'$ and $\mathcal{F}'= \lbrace{ \Lambda_{2},\dots,\Lambda_{|\mathcal{F}|}\rbrace}$. Then $(\mathcal{P}',\mathcal{F}',\mathcal{G}')$ is a counter-example such that $|\mathcal{F}'|=|\mathcal{F}|-1$ and $|\mathcal{G}'|\leq |\mathcal{G}|-1$. This contradicts minimality of $(\mathcal{P},\mathcal{F},\mathcal{G})$.
\end{proof}

\subsection{Proof of Proposition~\ref{prop:irreducible}}\label{sub:Reduction}
Let's first restate  Proposition~\ref{prop:irreducible} in the dual form.

\begin{prop}[Dual formulation] Let $\mathcal{A}$ be a simplicial arrangement of $n$ lines in $\mathbb{RP}^{2}$. Assume there exists an irreducible cubic $\hat \Gamma$ such that $k$ lines of $\mathcal{A}$ are {\rm tangent} to the dual cubic $\Gamma$. Then $n$ and $k$ satisfy $n \geq \frac{8}{7}k-3$.  
\end{prop}
We stress that in Proposition~\ref{prop:irreducible} and below a line of $\mathcal{A}$ is considered to be tangent to $\Gamma$ only if its dual point belongs to $\hat \Gamma$. We will split the proof of the proposition into five cases depending on the shape of the irreducible cubic $\hat\Gamma$ (see Section~\ref{sub:classification}).

\begin{proof}[Proof for a smooth irreducible cubic with one connected component]
In this case, we prove that $n \geq \frac{4}{3}k-3$. Let $A,B,C$ be the three cusps of $\Gamma$. We also define $\Delta$ as the curvilinear triangle cut out by $\Gamma$ (the component of $\mathbb{RP}^{2} \setminus \Gamma$ that has corner points of zero angle). Let $k_{AB}$ (resp. $k_{BC}$ and $k_{AC}$) be the number of lines of $\mathcal{A}$ that are tangent to the arc $AB$ of $\Gamma$ (resp. arcs $BC$ and $AC$), including lines tangent to $\Gamma$ at cusps $A$ and $B$ (resp. $B,C$ and $A,C$). Assuming without loss of generality that $k_{AB} \geq k_{BC} \geq k_{AC}$, we obtain that $k_{AB} \geq \frac{k}{3}$. We denote by $g$ the number of lines of $\mathcal{A}$ that are not tangent to $\Gamma$. If $k \leq 3$, then the linear bound $n \geq \frac{4}{3}k-3$ is trivially satisfied so we assume $k \geq 4$. It follows that $k_{AB} \geq 2$.
\begin{figure}[h!]
\includegraphics[scale=0.16]{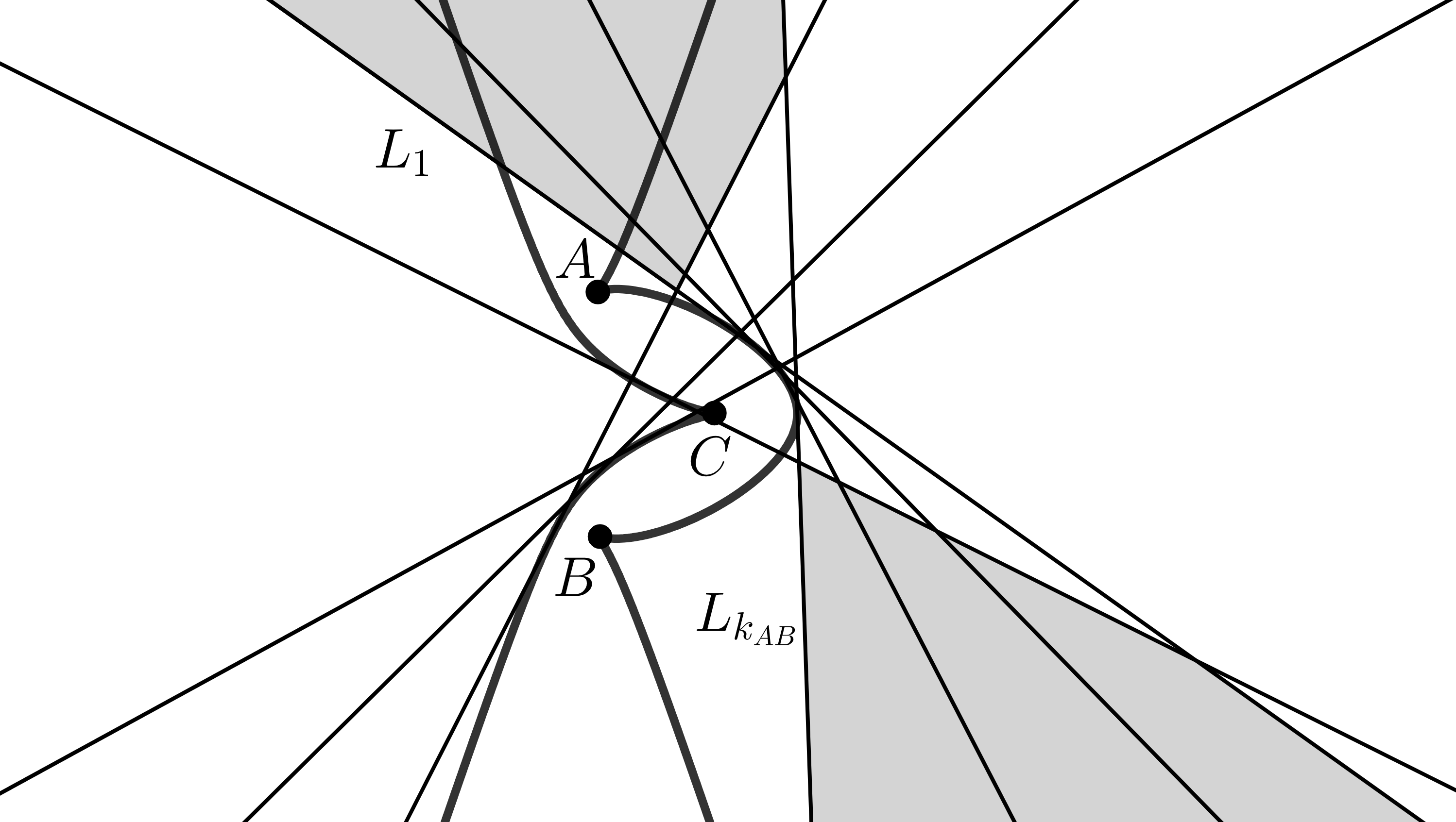}
\caption{The connected dual cubic and the lines of $\mathcal A$ tangent to it}
\label{fig:dualines}
\end{figure}
\par
Denote by $L_{1},\dots,L_{k_{AB}}$ the lines of $\mathcal{A}$ that are tangent to arc $AB$ according to the position of their tangency point in arc $AB$ (the tangency point of $L_{1}$ is the closest to $A$ while the tangency point of $L_{k_{AB}}$ is the closest to $B$). None of these lines contains the third cusp $C$. Indeed, if such a line $L$ contains $C$, the dual point $\hat{L}\in \hat\Gamma$ lies on the  line $\hat{C}$, tangent to $\hat{\Gamma}$ at the inflection point of $\hat\Gamma$ corresponding to cusp $C$. Hence $\hat C$ has intersection of multiplicity at least $4=3+1$ with $\hat \Gamma$, which contradicts B\'{e}zout's theorem. So lines $L_{1},\dots,L_{k_{AB}}$ do not contain $C$.
\par
Let $\mathcal{D}$ be the digon cut out in $\mathbb{RP}^{2}$ by $L_{1}$ and $L_{k_{AB}}$ that does not contain the cusp $C$. Denote by $\mathcal{B}$ the subarrangement of $\mathcal{A}$ formed by the lines that are tangent to arcs $BC$ and $AC$. Lines of $\mathcal{B}$ cut out a unique convex polygon $\mathcal{P}$ inside $\mathcal{D}$ that contains $\mathcal{D} \setminus \Delta$. $\mathcal{P}$ is shaded gray on Figure~\ref{fig:dualines}. By construction of $\mathcal P$, the lines of $\mathcal A$ tangent to $\Gamma$ that intersect the interior of $\mathcal{P}$ are precisely $L_{2},\dots,L_{k_{AB}-1}$. Since the arc $AB$ is convex, these $k_{AB}-2$ lines intersect among themselves outside $\mathcal{P}$. So Lemma~\ref{lem:CPInter} implies  $g \geq k_{AB}-3$.
\par
Since $n=k+g$ and $k_{AB} \geq \frac{k}{3}$, we obtain that $n \geq \frac{4}{3}k-3$.
\end{proof}

\begin{proof}[Proof for the cuspidal cubic]
We follow exactly the same reasoning as in the case where $\hat{\Gamma}$ is smooth and connected. This time $\Gamma$ is the union of the cuspidal cubic $y^2=x^3$ with the line at infinity. The curvilinear triangle $\Delta$ is formed by the union of the cuspidal cubic and the segment $x\ge 0$ on the line at infinity. The only difference with the first case is that the straight side of $\Delta$ is tangent to at most one line of $\mathcal A$ - the one containing it.
\end{proof}

\begin{proof}[Proof for a smooth irreducible cubic with two connected components]
In this case, we prove that  $n \geq \frac{8}{7}k -3$. We denote by $k_{\Delta}$ the number of lines of $\mathcal{A}$ that are tangent to the component with the three cusps while $k_{O}$ is the number of lines of $\mathcal{A}$ tangent to the oval. As before, $g$ is the number of lines in $\mathcal A$ not tangent to $\Gamma$.
\par
First, we prove that
\begin{equation}\label{eq:polygon}
g+k_{\Delta} \geq k_{O}-3.
\end{equation}
Let $\mathcal{O}$ be the subarrangement of $\mathcal{A}$ formed by the $k_{O}$ lines tangent to the oval of $\Gamma$. Among the faces cut out by the lines of $\mathcal{O}$, we denote by $\mathcal{R}$ the face containing $\Gamma$. If $k_{O} \geq 3$, $\mathcal{R}$ is a convex polygon with $k_{O}$ sides, and therefore has to be crossed by at least $k_{O}-3$ other lines of $\mathcal{A}$ to be triangulated. Inequality~\eqref{eq:polygon} follows. 
\par
Now, we settle the case $k_{\Delta} \leq 3$. Inequality~\eqref{eq:polygon} implies that $g \geq k_{O}-6$. Since $k=k_{O}+k_{\Delta}$, we obtain $g \geq k-9$ and therefore $n \geq 2k-9$ which is stronger than $n \geq \frac{8}{7}k -3$ provided that $k \geq 7$. If $k \leq 6$, inequality $n \geq \frac{8}{7}k -3$ automatically holds.
\par
\begin{figure}[h!]
\includegraphics[scale=0.16]{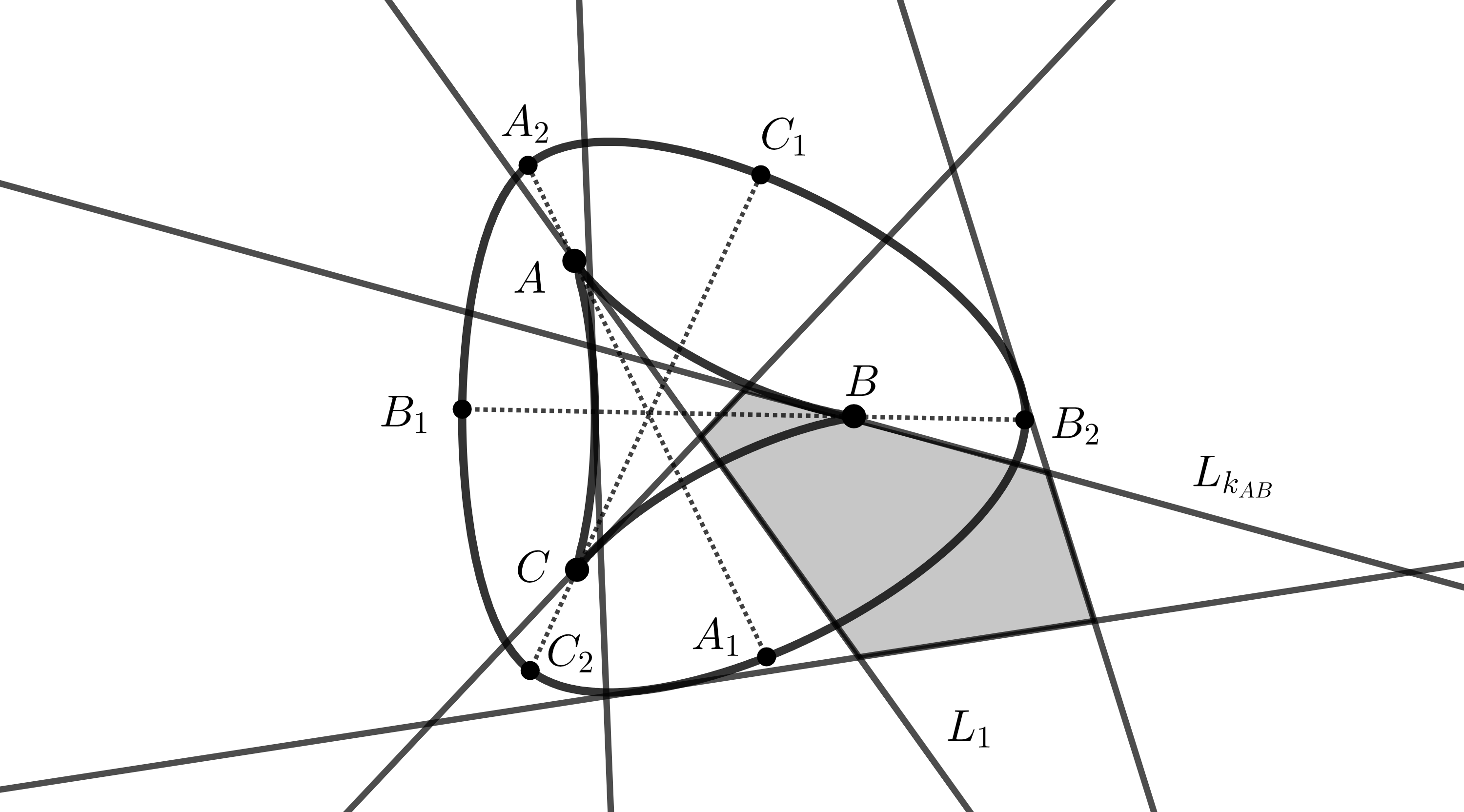}
\caption{The disconnected dual cubic and the lines of $\mathcal A$ tangent to it}
\label{fig:disconnectedlines}
\end{figure}

In the rest of the proof, we will assume that $k_{\Delta} \geq 4$. Following the same notations as in the case of a smooth irreducible cubic with one connected component, $k_{AB}$ denotes the number of lines of $\mathcal{A}$ that are tangent to the arc of the curvilinear triangle drawn between two cusps $A$ and $B$. Without loss of generality, we can assume that $k_{AB} \geq \frac{1}{3}k_{\Delta}$ and therefore $k_{AB} \geq 2$. Again, we denote by $\mathcal{D}$ the digon cut out by $L_{1}$ and $L_{k_{AB}}$ that does not contain the cusp $C$.
\par
Denote by $A_{1},A_{2}$ (resp. $B_{1},B_{2}$ and $C_{1},C_{2}$) the intersection of the oval of $\Gamma$ with the tangent line to $\Gamma$ at the cusp $A$ (resp. $B$ and $C$). Then, the oval decomposes into six arcs $(A_{1}C_{2}),(C_{2},B_{1}),(B_{1},A_{2}),(A_{2},C_{1}),(C_{1},B_{2}),(B_{2},A_{1})$. Observe that any line tangent to the arc of $\Delta$ with endpoint $A$ and $B$ intersects the oval two times in the arcs $(B_{1},A_{2})$ and $(B_{2},A_{1})$. 
\par
Let $\Theta$ be the intersection of the arc $(B_{2},A_{1})$ with the digon $\mathcal{D}$. Observe that the only lines of $\mathcal{A}$ that intersect $\Theta$ are $L_{2},\dots,L_{k_{AB}-1}$. Introduce the subarrangement $\mathcal{B} = \mathcal{A} \setminus \lbrace{ L_{2},\dots,L_{k_{AB}-1} \rbrace}$. Then, among the convex polygons cut out by the lines of $\mathcal{B}$, we define $\mathcal{P}$ as the convex polygon containing the arc $\Theta$. In particular, $\mathcal{P}$ is contained in $\mathcal{R}$. Polygon $\mathcal{P}$ is shaded in grey on Figure~\ref{fig:disconnectedlines}.
%\begin{figure}[h!]
%\includegraphics[scale=0.16]{???}
%\caption{The dual cubic and the lines of $\mathcal A$ tangent to it}
%\label{fig:dualinesTWOCOMPONENTS}
%\end{figure}
\par
Lines $L_{2},\dots,L_{k_{AB}-1}$ intersect the interior of $\Theta$ and therefore the interior of $\mathcal{P}$. Besides, as they intersect outside of $\mathcal{D}$, they intersect outside of $\mathcal{P}$. Lemma~\ref{lem:CPInter} proves that $g \geq k_{AB}-3$ and thus
\begin{equation}\label{eq:seconde}
g \geq \frac{1}{3}k_{\Delta}-3.    
\end{equation}
Adding \eqref{eq:polygon} to \eqref{eq:seconde} multiplied by $6$, we get $7g \geq k_{\Delta} + k_{O} -21=k-21$. Using  $g=n-k$, we finally get $n \geq \frac{8}{7}k -3$.
\end{proof}

\begin{proof}[Proof for an acnodal irreducible cubic]
The proof and the inequality are the same as in the case of a smooth irreducible cubic with two connected components. We just have to add the constraint $k_{O} \leq 1$.
\end{proof}

\begin{proof}[Proof for a nodal irreducible cubic]
Let $A,B,C$ be respectively the two tangent points of the unique bitangent line to $\Gamma$ and its unique cusp. We denote by $k_{O}$ the number of lines that are tangent to the arc of $\Gamma$ between $A$ and $B$ (including possibly the bitangent line $AB$). Then, we denote by $k_{\Delta}$ the number of other lines of $\mathcal{A}$ that are tangent to $\Gamma$ (in the two arcs between $A,B$ respectively and the cusp $C$). We can now carry out the same reasoning as in the case of a smooth irreducible cubic with two connected components (with the same notations). We obtain the same inequality.
\end{proof}

\section{Proof of the main theorem}\label{sec:rigidity}

\subsection{Geometry of regular simplicial arrangements}\label{sub:trigo}

A regular simplicial arrangement is the union of two subarrangements defined with respect to a circle $\gamma$ which is, for our purpose, the unit circle of the complex plane:
\begin{itemize}
    \item $\mathcal{A}_{\gamma}$ formed by the lines $L_{1},\dots,L_{m}$ that are tangent to $\gamma$ at the $m^{th}$ roots of the unit $e^{\frac{2ij\pi}{m}}$ where $j \in \mathbb{Z}/m\mathbb{Z}$;
    \item a pencil $\mathcal{A}_{C}$ formed by the $m$ lines containing the centre of circle $\gamma$ and a $2m^{th}$ root of the unit $e^{\frac{ij\pi}{m}}$ where $j \in \mathbb{Z}/m\mathbb{Z}$;. 
\end{itemize}

We introduce the notation $Z_{j,l}$ for the complex coordinate of vertex $L_{j} \cap L_{l}$. A direct trigonometric computation proves that $|Z_{j,l}|=\frac{1}{\cos(|j-l|\pi/m)}$ while $\arg(Z_{j,l})=\frac{(j+l)\pi}{m}$.
\par
In particular, with the possible exception of the line at infinity that contains $\frac{m}{2}$ vertices if $m$ is even, the vertices of $\mathcal{A}_{\gamma}$ are contained in concentric circles of radii $\frac{1}{\cos(|j-l|\pi/m)}$ containing $m$ vertices each. 
\par
The $m$ lines of line arrangement $\mathcal{A}_{\gamma}$ cut out (see Lemma~\ref{lem:centralpolygon}): 
\begin{itemize}
\item a $m$-gon $\mathcal{P}$ containing circle $\gamma$ and called the \textit{central polygon};
\item $m$ triangles;
\item $\frac{m(m-3)}{2}$ quadrilaterals.
\end{itemize}
The vertices of the central polygon $\mathcal{P}$ belong to the circle of radius $\frac{1}{\cos(\pi/m)}$. The $m$ triangles are adjacent to $\mathcal{P}$ and share two vertices with it. The third vertex of each triangle belongs to the circle of radius $\frac{1}{\cos(2\pi/m)}$.
\par
We prove now that some triples of vertices of $\mathcal{A}_{\gamma}$ cannot belong to a same line.

\begin{lem}\label{lem:alignement}
In the line arrangement $\lbrace{ L_{1},\dots, L_{m} \rbrace}$ defined by the sides of a regular $m$-gon with $m \geq 3$, for any $j,l \in \mathbb{Z}/m\mathbb{Z}$ such that $|l-j| \geq 4$, the three points $L_{j} \cap L_{l}$, $L_{j+1} \cap L_{l-2}$ and $L_{j-1} \cap L_{l+2}$ cannot be on a same line. Similarly, $L_{j} \cap L_{l}$, $L_{j+2} \cap L_{l-1}$ and $L_{j-2} \cap L_{l+1}$ cannot be on a same line either.
\end{lem}

\begin{proof}
Unless $m \geq 8$, the statement is empty. We first prove that statement for the three points $L_{j} \cap L_{l}$, $L_{j+1} \cap L_{l-2}$ and $L_{j-1} \cap L_{l+2}$ and will denote them respectively by $B,A,C$. The statement for triples of the form $L_{j} \cap L_{l}$, $L_{j+2} \cap L_{l-1}$ and $L_{j-2} \cap L_{l+1}$ is deduced from the previous one by a symmetry argument.
\par
Suppose $A,B,C$ belong to the same line and none of them is at infinity. Then considering the triangle $AOC$, we get the following expression for the length of bisector $OB$
$$OB=\frac{2 \cos(\theta) \cdot OA \cdot OC}{OA + OC}$$
where $\theta = \frac{\pi}{m}$ and $O$ is the center $O$ of the regular polygon $\mathcal{P}$.
\par
In our case, we have $\theta = \frac{\pi}{m}$, $OA = \frac{1}{\cos((k-3)\theta)}$, $OB=\frac{1}{\cos(k\theta)}$ and $OC=\frac{1}{\cos((k+3)\theta)}$ where $k =|j-l|$. We obtain
$$
\cos((k-3)\theta) + \cos((k+3)\theta) = 2\cos (\theta) \cdot \cos(k\theta).
$$
while the correct trigonometric identity is
$$
\cos((k-3)\theta) + \cos((k+3)\theta) = 2\cos (3\theta) \cdot \cos(k\theta).
$$
By hypothesis, $B$ is not a point at infinity so $\cos(k\theta) \neq 0$. Besides, $m \geq 8$ so $\cos(\theta)<\cos(3\theta)$. The two trigonometric identities cannot hold simultaneously so $A$, $B$ and $C$ cannot be colinear. When $C$ (or equivalently $A$) is a point at infinity, the same trigonometric identity holds and we get a contradiction in the same way.
\par
In the last case, $B$ is a point at infinity. This corresponds to the case where $B=L_{0} \cap L_{m/2}$ (up to a rotation). $B$ is the point at infinity in the vertical direction while $A$ and $C$ are respectively $L_{1} \cap L_{(m-4)/2}$ and $L_{-1} \cap L_{(m+4)/2}$. We check immediately that the slope of the line between these two points is not vertical.
\end{proof}

We deduce from Lemma~\ref{lem:alignement} that some configurations of colinear points cannot exist in a regular simplicial arrangement.

\begin{prop}\label{prop:configuration}
Consider the simplicial arrangement $\mathcal{A}_{r}$ formed by the sides and symmetric axes of a regular $m$-gon $\mathcal{P}$ for $m \geq 3$ ($\mathcal{A}_{r}$ belongs to the combinatorial class $\mathcal{A}(2m,1)$). If $\Lambda$ is a line such that:
\begin{itemize}
    \item $\Lambda \notin \mathcal{A}_{r}$;
    \item $\Lambda$ does not contain the centre $C$ of the regular $m$-gon $\mathcal{P}$;
    \item $\Lambda$ is not the line at infinity (for the Euclidean structure where $\mathcal{P}$ is a regular polygon);
\end{itemize}
then $\Lambda$ does not contain a configuration of five consecutive intersection points $N_{1},N_{2},N_{3},N_{4},N_{5}$ with $\mathcal{A}_{r}$ such that:
\begin{itemize}
    \item $N_{1},N_{3},N_{5}$ are triple points of $\mathcal{A}_{r}$;
    \item $N_{2},N_{4}$ are interior points of edges of $\mathcal{A}_{r}$.
\end{itemize}
\end{prop}

\begin{proof}
We assume by contradiction that $\Lambda$ contains such a configuration $N_{1},N_{2},N_{3},N_{4},N_{5}$ of consecutive intersection points with $\mathcal{A}_{r}$. We decompose $\mathcal{A}_{r}$ into two subarrangements $\mathcal{A}_{\gamma}$ and $\mathcal{A}_{C}$ where $\mathcal{A}_{\gamma}$ is formed by the sides of $\mathcal{P}$ while $\mathcal{A}_{C}$ is the pencil formed by the lines of $\mathcal{A}_{r}$ that contains the centre $C$ (these lines are also the symmetry axes of $\mathcal{P}$).
\par
Since $N_{1},N_{3},N_{5}$ are triple points in $\mathcal{A}_{r}$, they are vertices of the subarrangement $\mathcal{A}_{\gamma} \subset \mathcal{A}_{r}$. The two lines of $\mathcal{A}_{\gamma}$ intersecting in $N_{3}$ cut out four angular sectors, two of which corresponding to lines incident to $N_{3}$ intersecting the interior of polygon $\mathcal{P}$. We will refer to them as radial sectors (the two other angular sectors will be referred to as the nonradial sectors).
\par
We first assume that $\Lambda$ intersects the nonradial sectors of $N_{3}$. If these sectors belong to triangles of arrangement $\mathcal{A}_{\gamma}$ (see Lemma~\ref{lem:centralpolygon}), then $N_{2}$ lies on a symmetry axis of the triangle while $N_{1}$ lies on a boundary edge of the triangle and therefore cannot be a triple point of $\mathcal{A}_{r}$. If these nonradial sectors belong to quadrilaterals of arrangement $\mathcal{A}_{\gamma}$, then $N_{2}$ is contained in an edge of $\mathcal{A}_{r}$ contained in a line of $\mathcal{A}_{C}$ and $N_{1}$ belongs to the same concentric circle as $N_{3}$ (we have |$CN_{1}|=|CN_{3}|$). Then $\Lambda$ belongs to a non-radial sector of $N_{3}$ and $N_{4}$ plays a similar role for $N_{3},N_{5}$ as $N_{2}$ for $N_{1},N_{3}$. It follows that vertices $N_{1},N_{3},N_{5}$ belong to a same circle centered on $C$. Therefore, the only case where $N_{1},N_{3},N_{5}$ are colinear is when they are points at infinity for this affine chart. This case is ruled out by hypothesis.
\par
In the remaining case, $\Lambda$ intersects the radial sectors of $N_{3}$ (it follows from the previous arguments that $\Lambda$ also intersects the radial sectors of $N_{1}$ and $N_{5}$). We check that if a radial sector of $N_{3}$ intersected by $\Lambda$ belongs to a triangle of arrangement $\mathcal{A}_{\gamma}$, then $N_{2}$ (or $N_{4}$ depending on the orientation we chose) belongs to the edge of this triangle that is incident to polygon $\mathcal{P}$ and $N_{1}$ (or $N_{5}$) cannot be a vertex of $\mathcal{A}_{\gamma}$. The same problem appears if the radial sector of $N_{3}$ containing $\Lambda$ belongs to polygon $\mathcal{P}$. Hence the radial sectors of $N_{3}$ intersected by $\Lambda$ are contained in quadrilaterals of arrangement $\mathcal{A}_{\gamma}$. The arguments given previously show similarly that the four edges $[N_{1},N_{2}]$, $[N_{2},N_{3}]$, $[N_{3},N_{4}]$ and $[N_{4},N_{5}]$ are contained in quadrilaterals of $\mathcal{A}_{\gamma}$. 
\par
Assuming that $N_{3}$ is the intersection point of $L_{j}$ and $L_{l}$ (where $L_{1},\dots,L_{m}$ are the lines of $\mathcal{A}_{\gamma}$ endowed with the cyclic order induced by $\mathcal{P}$), then one of the two following statements holds:
\begin{itemize}
    \item $N_{1}$ and $N_{5}$ are $L_{j+1} \cap L_{l-2}$ and $L_{j-1} \cap L_{l+2}$;
    \item $N_{1}$ and $N_{5}$ are $L_{j+2} \cap L_{l-1}$ and $L_{j-2} \cap L_{l+1}$;
\end{itemize}
where $|j-l| \geq 4$ (otherwise, $[N_{1},N_{5}]$ would intersect the interior of a triangle of $\mathcal{A}_{r}$ or the interior of $\mathcal{P}$). Lemma~\ref{lem:alignement} rules out these cases. We conclude that there is no such configuration $N_{1},N_{2},N_{3},N_{4},N_{5}$.
\end{proof}

\subsection{Arrangements differing from a regular arrangement by a small proportion of lines}\label{sub:difference}

Drawing on the description of regular simplicial arrangements given in Section~\ref{sub:trigo}, we prove in this section that a simplicial arrangement differing by a small proportion of lines from a regular simplicial arrangement automatically belongs to family $\mathcal{R}(1)$ or $\mathcal{R}(2)$.

\begin{prop}\label{prop:rigidity}
Consider a simplicial arrangement $\mathcal{A}$ of $n \geq 15$ lines in $\mathbb{RP}^{2}$ that differs by at most $k < \frac{1}{13}n - \frac{14}{13}$ lines from a regular simplicial arrangement $\mathcal{A}_{r}$ formed by the sides and symmetry axes of a regular $m$-gon $\mathcal{P}$ ($\mathcal{A}_{r}$ belongs to the combinatorial class $\mathcal{A}(2m,1)$). Then $\mathcal{A}$ coincides either with $\mathcal{A}_{r}$ or with the union of $\mathcal{A}_{r}$ with the line at infinity.
\end{prop}

First, we give a lemma ruling out additional lines containing the center of the regular arrangement.

\begin{lem}\label{lem:rigidity}
Consider a simplicial arrangement $\mathcal{A}$ of $n \geq 15$ lines in $\mathbb{RP}^{2}$ that differs by at most $k < \frac{1}{7}n + \frac{1}{7}$ lines from a regular simplicial arrangement $\mathcal{A}_{r}$ formed by the sides and symmetry axes of a regular $m$-gon $\mathcal{P}$ ($\mathcal{A}_{r}$ belongs to the combinatorial class $\mathcal{A}(2m,1)$). Then no line $\Lambda$ of $\mathcal{A} \setminus \mathcal{A}_{r}$ contains the centre $C$ of the regular polygon $\mathcal{P}$. Besides, $\mathcal{A}$ cannot be a near-pencil.
\end{lem}

\begin{proof}
By our assumptions, $\frac{n-k}{2} \leq m \leq \frac{n+k}{2}$. Also arrangement $\mathcal{A}$ contains at least $m-k$ sides of polygon $\mathcal{P}$. Since $m-k \geq \frac{n-3k}{2} > \frac{2}{7}n - \frac{3}{14} \geq 4$ (because $k < \frac{1}{7}n+\frac{1}{7}$ and $n \geq 15$), there is a polygon with at least four sides formed by the union of several faces of $\mathcal{A}$. Hence it cannot belong to family $\mathcal{R}(0)$.
\par
Assume by contradiction there exists a line $\Lambda \in \mathcal{A} \setminus \mathcal{A}_{r}$ containing the centre $C$ of the regular polygon $\mathcal{P}$. Let $M_{1},\dots,M_{t}$ be a cyclically ordered set of points on $\Lambda$ where $M_{1}=C$ while $M_{2},\dots,M_{t}$ are the intersection points (other than $C$) of $\Lambda$ with the other lines of the symmetric difference $\mathcal{A} \Delta \mathcal{A}_{r}$.  We have $t \leq k$ since $|\mathcal{A} \Delta \mathcal{A}_{r}| \leq k$ and $\Lambda \in \mathcal{A} \Delta \mathcal{A}_{r}$. A line containing $C$ cannot contain any other vertex of $\mathcal{A}_{r}$ without being itself a line of $\mathcal{A}_{r}$. Therefore, $\Lambda$ crosses at most one line of $\mathcal{A}_{r}$ at each of points of $M_{2},\dots,M_{t}$. To sum up, $\Lambda$ intersects $m$ lines of $\mathcal{A}_{r}$ in $C=M_{1}$, at most $t-1$ lines of $\mathcal{A}_{r}$ in $M_{2},\dots,M_{t}$ and $k-1$ lines of $\mathcal{A} \Delta \mathcal{A}_{r}$ in $M_{2},\dots,M_{t}$. It follows that $\Lambda$ intersects at most $m+t+k-2$ lines of $\mathcal{A}$ in $M_{1},\dots,M_{t}$. 
\par
By hypothesis, inside any open interval  $]M_{i},M_{i+1}[$, line $\Lambda$ only intersects lines of $\mathcal{A}_{r}$. Besides, it cannot contain any vertex of $\mathcal{A}_{r}$ (because $C=M_{1}$ and $\Lambda$ cannot contain any other vertex of $\mathcal{A}_{r}$ without being a line of $\mathcal{A}_{r}$). It follows that every vertex of $]M_{i},M_{i+1}[$ is a double point of $\mathcal{A}$. We know from Lemma~\ref{lem:adjaDP} that two double points cannot be adjacent in a simplicial arrangement that does not belong to family $\mathcal{R}(0)$. It follows that in any open interval $]M_{i},M_{i+1}[$, $\Lambda$ intersects at most one line of $\mathcal{A}$.
\par
We conclude that line $\Lambda$ intersects at most $m+t+k-2$ lines of $\mathcal{A}$ in $M_{1},\dots,M_{t}$ and at most $t$ lines of $\mathcal{A}$ outside $M_{1},\dots,M_{t}$. Thus, we have $n-1 \leq m+2t+k-2 \leq \frac{1}{2}n+\frac{7}{2}k-2$ (because $t \leq k$ and $m \leq \frac{n+k}{2}$). It follows that $\frac{1}{7}n + \frac{1}{7} \leq k$. This contradicts our assumptions, and so there is no $\Lambda \in \mathcal{A} \setminus \mathcal{A}_{r}$ containing $C$.
\end{proof}

A variant of the above argument gives further restrictions on lines added to  is a regular arrangement. We will see that the only line that can be added to $\mathcal{A}_{r}$  is the line at infinity.

\begin{lem}\label{lem:rigidity1}
Consider a simplicial arrangement $\mathcal{A}$ of $n \geq 15$ lines in $\mathbb{RP}^{2}$ that differs by at most $k < \frac{1}{13}n - \frac{14}{13}$ lines from a regular simplicial arrangement $\mathcal{A}_{r}$ formed by the sides and symmetric axes of a regular $m$-gon $\mathcal{P}$ ($\mathcal{A}_{r}$ belongs to the combinatorial class $\mathcal{A}(2m,1)$). Then $\mathcal{A} \setminus \mathcal{A}_{r}$ contains at most one line which is the line at infinity (for an Euclidean structure where regular polygon $\mathcal{P}$ is inscribed in a circle).
\end{lem}

\begin{proof}
Consider a line $\Lambda \in \mathcal{A} \setminus \mathcal{A}_{r}$. Conditions of Lemma~\ref{lem:rigidity1} are more restrictive than that of Lemma~\ref{lem:rigidity} so we can assume that $\Lambda$ does not contain the center $C$. We also assume by contradiction that $\Lambda$ is not the line at infinity for an Euclidean structure where regular polygon $\mathcal{P}$ is inscribed in a circle.
\par
Let $M_{1},\dots,M_{t}$ be the cyclical ordered subset of $\Lambda$ consisting of:
\begin{itemize}
    \item[a)] the intersection points of $\Lambda$ with the lines of the symmetric difference $\mathcal{A} \Delta \mathcal{A}_{r}$;
    \item[b)] the double points of arrangement $\mathcal{A}_{r}$ possibly contained in $\Lambda$.
\end{itemize}
%These points are ordered according to the cyclic order of $\Lambda$. 
Let's show first that $t \leq k+1$.  Indeed, line $\Lambda$ has at most $k-1$ intersection points with the lines of $\mathcal{A} \Delta \mathcal{A}_{r}$, so there are at most $k-1$ points of type a) among $M_i$'s. And at most $2$ of $M_i$'s are of type b), since  $\mathcal A_r$ has exactly $2m$ double points, namely the midpoints of edges of the convex polygon $\mathcal{P}$.
%Besides, there is exactly one double point in the interior of each exterior edge of convex polygon $\mathcal{P}$ (and no double point elsewhere). 
%Therefore, $\Lambda$ cannot contain more than two of them. It follows that $t \leq k+1$. 

Next, we claim that there are at most $4t$ lines of ${\mathcal A}\setminus \Lambda$ that are incident to $M_1,\ldots, M_t$. Indeed, since $C\notin \Lambda$, each $M_i$ is at most a triple point of $\mathcal A_r$,  and $|\mathcal{A} \Delta \mathcal{A}_{r}\setminus \Lambda|=k-1$ . Taking the sum of contributions, and using   $t \leq k+1$ we get $3t+k-1\le 4t$.
 
%(indeed, a triple point of $\mathcal{A}_{r}$ can still be one of the $M_{1},\dots,M_{t}$ if it is an intersection point with one line of $\mathcal{A} \Delta \mathcal{A}_{r}$).
\par
Our goal will now be to prove that each open interval $]M_{i},M_{i+1}[$ of $\Lambda$ intersects at most $9$ lines of $\mathcal{A}\setminus \Lambda$. If the set $\lbrace{ M_{1},\dots,M_{t} \rbrace}$ is empty, we will prove similarly that the whole projective line intersects at most $9$ lines of $\mathcal{A}\setminus \Lambda$. Then, counting the lines of $\mathcal{A}\setminus\Lambda$ intersected by $\Lambda$ inside and outside of $M_{1},\dots,M_{t}$, we obtain that $n-1 \leq 4t+9t \leq 13k+13$ (since $t \leq k+1$). This contradicts the inequality $k < \frac{1}{13}n - \frac{14}{13}$ and $\Lambda$ is therefore the line at infinity.\newline

It remains to prove that in each open interval $]M_{i},M_{i+1}[$ (or the whole line if $\lbrace{ M_{1},\dots,M_{t} \rbrace}$ is empty), $\Lambda$ intersects at most $9$ lines of $\mathcal{A} \setminus \lbrace{ \Lambda \rbrace}$. We assume by contradiction that there is such an open interval (that can be the whole line) $]M_{i},M_{i+1}[$ in $\Lambda$ that intersects at least $10$ lines of $\mathcal{A}$. Since the intersections of $\Lambda$ with the lines of $\mathcal{A} \Delta \mathcal{A}_{r}$ are contained in $M_{1},\dots,M_{n}$, all these lines of $\mathcal{A}$ that $\Lambda$ intersects in $]M_{i},M_{i+1}[$ actually are lines of  $\mathcal{A}_{r}$. It also follows from the definition of $M_{1},\dots,M_{t}$ that every line of $\mathcal{A}_{r}$ intersecting $\Lambda$ outside of $M_{1},\dots,M_{t}$ actually belongs to $\mathcal{A}$.
\par
The only vertices of $\mathcal{A}_{r}$ that line $\Lambda$ can meet outside of $M_{1},\dots,M_{t}$ are triple points of $\mathcal{A}_{r}$. Therefore, the vertices of $\mathcal{A}$ contained in the open segment $]M_{i},M_{i+1}[$ are either quadruple points of $\mathcal{A}$ (if they are vertices of $\mathcal{A}_{r}$) or double points of $\mathcal{A}$ (otherwise).
\par
Arrangement $\mathcal{A}$ does not belong to family $\mathcal{R}(0)$ (see the last claim in Lemma~\ref{lem:rigidity}). Hence, there cannot be two consecutive double points in $]M_{i},M_{i+1}[$ (see Lemma~\ref{lem:adjaDP}). We prove now that there cannot be two consecutive quadruple points $A,B$ of $\mathcal{A}$ in $]M_{i},M_{i+1}[$. Indeed, such segment $[A,B] \subset ]M_{i},M_{i+1}[$ is drawn between two vertices of $\mathcal{A}_{r}$ so unless it is contained in a line of $\mathcal{A}_{r}$ (and $\Lambda \notin \mathcal{A}_{r}$), $S$ intersects the interior of an edge of $\mathcal{A}_{r}$. This is a contradiction with the statement that $A,B$ are consecutive intersection points with $\mathcal{A}_{r}$ in $]M_{i},M_{i+1}[$. So vertices of segment $]M_{i},M_{i+1}[$ alternate between double and quadruple points. Since $]M_{i},M_{i+1}[$ intersects at least $10$ lines of $\mathcal{A}_{r}$, we conclude that the open segment $]M_{i},M_{i+1}[$ contains a sequence of five consecutive vertices $N_{1},N_{2},N_{3},N_{4},N_{5}$ where $N_{1},N_{3},N_{5}$ are quadruple points while $N_{2},N_{4}$ are double points.
\par
Since $N_{1},N_{3},N_{5}$ are triple points in $\mathcal{A}_{r}$ while $N_{2},N_{4}$ are interior points of edges in $\mathcal{A}_{r}$, Proposition~\ref{prop:configuration} proves that there is no line containing this configuration of five points. This final contradiction ends the proof.
\end{proof}

We will prove now that no simplicial arrangement can be obtained by removing a small proportion of lines from a regular simplicial arrangement (or the arrangement obtained by adding to it the line at infinity), finishing thus the proof of Proposition~\ref{prop:rigidity}.

\begin{proof}[Proof of Proposition~\ref{prop:rigidity}]
Drawing on the conclusions of Lemma~\ref{lem:rigidity1}, $\mathcal{A}$ is obtained by removing some lines from an arrangement $\mathcal{A}_{max}$ which is either $\mathcal{A}_{r}$ or the arrangement formed by $\mathcal{A}_{r}$ and the line at infinity.
\par
Assume first by contradiction that there is a line $\Lambda\subset \mathcal{A}_{\max} \setminus \mathcal{A}$  incident to the centre $C$ of the polygon $\mathcal{P}$. 
Decompose $\mathcal{A}_{r}$ into the pencil $\mathcal{A}_{C}$ of $m$ lines incident to $C$ and the arrangement $\mathcal{A}_{\gamma}$ of $m$ lines forming the sides of $\mathcal{P}$. Observe that the number of quadrilaterals of $\mathcal{A}_{\gamma}$ with one diagonal  contained in $\Lambda$ (see Lemma~\ref{lem:centralpolygon}) is given by:
\begin{itemize}
    \item[(i)] $\frac{m-2}{2}$ if $m$ is even and $\Lambda$ contains two opposite corners of $\mathcal{P}$;
    \item[(ii)] $\frac{m-4}{2}$ if $m$ is even and $\Lambda$ crosses the interior of two triangles of $\mathcal{A}_{\gamma}$;
    \item[(iii)] $\frac{m-3}{2}$ if $m$ is odd.
\end{itemize}
Furthermore, in case (i) the interior of one of quadrilaterals is crossed by the line at infinity, in case (ii) the interior of none of quadrilaterals is crossed by the line at infinity, and in case (iii) the interior of one of quadrilaterals is crossed by the line at infinity.

We conclude that among  quadrilaterals crossed by $\Lambda$, at least $\frac{m-5}{2}$  have the interior disjoint from the line at infinity. Since $\mathcal{A}$ is simplicial, each of these quadrilaterals is contained in a triangle of $\mathcal{A}$, so at least one of their sides should also be removed from $\mathcal{A}_{max}$. Since the same line can contain a side of at most two of these quadrilaterals, the total number of lines of $\mathcal{A}_{\max} \setminus \mathcal{A}$ is at least $\frac{m-5}{4}$. We obtain $k \geq \frac{m-5}{4}$. Since we have $n \leq 2m+1$, we get $k \geq \frac{n}{8}-\frac{11}{8}$. Since $n \geq 15$, this contradicts $k < \frac{1}{13}n - \frac{14}{13}$.
\par
We conclude that the the center of polygon $\mathcal{P}$ is incident to exactly $m$ lines of $\mathcal{A}$. It follows from Proposition~\ref{prop:star1} that the star $S(C)$ of $C$ in the simplicial arrangement $\mathcal{A}$ is cut out by at least $m$ lines (Lemma~\ref{lem:rigidity} proves that $\mathcal{A}$ cannot belong to family $\mathcal{R}(0)$). We deduce that $\mathcal{A}$ coincides with $\mathcal{A}_{max}$.
\end{proof}

\subsection{Proof of the main theorem}\label{sub:PROOF}

Combining the results from Sections~\ref{sub:star},~\ref{sec:estimates} and~\ref{sec:rigidity}, we finally prove the asymptotic classification of simplicial arrangements under the additional hypothesis of a linear bound on the number of double points.

\begin{proof}[Proof of Theorem~\ref{thm:MAIN}]
Assuming that a simplicial arrangement $\mathcal{A}$ of $n$ lines contains at most $Kn$ double points (for some $K>1$) and that $n \geq  exp~exp (cK^{c})$, Theorem~\ref{thm:GT} proves that $\mathcal{A}$ differs (up to a projective transformation) by at most $O(K)$ lines from one of the three models described in the statement of the theorem.
\par
In the first case, $\mathcal{A}$ differs by a small proportion of lines from a pencil. In particular, provided constant $c$ is large enough, $\mathcal{A}$ contains a vertex $C$ of order $k_{C}$ arbitrarily close to $n$. Proposition~\ref{prop:star1} proves that if $\mathcal{A}$ is not a simplicial arrangement of family $\mathcal{R}(0)$, the star $S(C)$ of $C$ is cut out by at least $k_{C}$ lines that do not contain $C$ and $k_{C} \leq \frac{n}{2}$. It follows that $\mathcal{A}$ belongs to family $\mathcal{R}(0)$.
\par
In the second case, we deduce from Proposition~\ref{prop:rigidity} that $\mathcal{A}$ belongs to family $\mathcal{R}(1)$ or family $\mathcal{R}(2)$.
\par
In the third case, an arbitrarily large proportion of lines of $\mathcal{A}$ is tangent to the projective dual curve of an irreducible cubic. Proposition~\ref{prop:irreducible} proves that this is impossible (provided $c$ is large enough).
\end{proof}

\appendix

\section{Projective rigidity of regular simplicial arrangements}

As we use Green-Tao Theorem~\ref{thm:GT}, the asymptotic classification we obtain is not merely combinatorial but works at the level of projective equivalence classes. However, the fact that the combinatorial and projective equivalence relations coincide for the families $\mathcal{R}(1)$ and $\mathcal{R}(2)$ is valid in general and not just asymptotically. It follows from a result of elementary projective geometry.

\begin{prop}\label{prop:chareg}
Let $\mathcal{P} \subset \mathbb{R}^{2} \subset \mathbb{RP}^{2}$ be a convex $n$-gon with $n\geq 5$ vertices $x_{1},\ldots,x_{n}$. Suppose there is a point $C$ in the interior of $\mathcal{P}$ such that for any $i$ the following conditions hold:
\begin{itemize}
    \item Points $C$, $x_i$, and $x_{i-2}x_{i-1}\cap x_{i+1}x_{i+2}$ are aligned;
    \item Points $C$, $x_{i-2}x_{i-1}\cap x_{i}x_{i+1}$, and  $x_{i-3}x_{i-2}\cap x_{i+1}x_{i+2}$ are aligned.
\end{itemize}
Then there exists a projective transformation of $\mathbb{RP}^{2}$ that sends $\mathcal{P}$ to the regular $n$-gon and $C$ to its centre.
\end{prop}

\begin{proof}
We consider the dual configuration in the dual projective plane. Dual polygon $\widehat{\mathcal{P}}$ is a convex polygon cut out by dual lines $\hat{x}_{1},\ldots, \hat{x}_{n}$ while $\widehat{C}$ is the dual line of point $C$. Besides, all the points of line $\widehat{C}$ are disjoint from $\widehat{\mathcal{P}}$.
\par
We consider an affine chart of the dual projective plane where $\widehat{C}$ is the line at infinity. Denoting by $y_{1},\ldots, y_{n}$ of $\widehat{\mathcal{P}}$ where $y_{i}=\hat{x}_{i} \cap \hat{x}_{i+1}$, the hypotheses on polygon $\mathcal{P}$ correspond to the following hypotheses on the dual polygon $\widehat{\mathcal{P}}$:
\begin{itemize}
    \item for any $i$, the side $y_iy_{i+1}$ is parallel to the diagonal $y_{i-1}y_{i+2}$;
    \item the diagonal $y_{i-1}y_{i+1}$ is parallel to $y_{i-2}y_{i+2}$.
\end{itemize}
We will prove that $\widehat P$ is affinely equivalent to a regular $n$-gon, proving in consequence that configuration $(\mathcal{P},C)$ is projectively equivalent to a configuration formed by a regular $n$-gon and its centre.
\par
Consider the unique conic $\widehat{Q}$ containing points $y_1,\ldots, y_5$. Assume first that this conic is an ellipse. Up to an affine transformation, we can assume that $\widehat{Q}$ is a circle. Now, since $y_1y_5$ and $y_2y_4$ are parallel, we have $|y_1y_2|=|y_4y_5|$. Since $y_2y_5$ and $y_3y_4$ are parallel, we have $|y_2y_3|=|y_4y_5|$, and  since $y_1y_4$ and $y_2y_3$ are parallel, we have $|y_1y_2|=|y_3y_4|$. We conclude that the four sides $y_1y_2,\ldots, y_4y_5$ have equal length. Now, point $y_6$ is the intersection of diagonals $y_6y_2$ and $y_6y_3$ which are parallel to $y_5y_3$ and $y_5y_4$ correspondingly. It follows that $y_6$ lies on $\widehat{Q}$ and $|y_5y_6|=|y_5y_4|$. Repeating this argument we prove that vertices $y_1,\ldots,y_{n}$ lie on $\widehat{Q}$ and the polygon is regular.
\par
Let us now prove that the conic $Q$ cannot be neither a parabola, nor a hyperbola. Assume by contradiction that $\widehat{Q}$ is a hyperbola. Since $\widehat P$ is convex and has at least five vertices, it lies on one of the branches of $\widehat{Q}$, which we denote $\widehat{Q}_{+}$. Now, there is an action of the additive group $\mathbb{R}$ on the affine plane $\mathbb{R}^{2}$ by affine transformations, that  preserves $Q$ and is transitive on $\widehat{Q}_{+}$. Denote by $dt$ an $\mathbb R$-invariant $1$-form on $\widehat{Q}_{+}$ (unique up to scale) and let $t$ be the corresponding parametrization of $\widehat{Q}_{+}$. Note that for any four points $x_1,x_2,x_3,x_4$ on $\widehat{Q}_{+}$ such that $x_1x_4$ is parallel to $x_2x_3$ we have $t(x_2)-t(x_1)=t(x_4)-t(x_3)$. Then, applying the same reasoning as in the case when $\widehat{Q}$ is an ellipse, we get $t(y_2)-t(y_1)=\ldots=t(y_5)-t(y_4)$. In other words, there is an element $\varphi\in \mathbb R$ such that $\varphi(y_i)=y_{i+1}$ for $1 \leq i \leq 4$. Repeating the argument from the case where $\widehat{Q}$ is an ellipse we see that $y_6\in \widehat{Q}_{+}$ and so on until $y_n$. However, now we get a contradiction, since $y_{n+1}=y_1$ should satisfy $y_{n+1}=\varphi(y_n)$ by the same reasoning, which is clearly not the case.
\par
The case where $\widehat{Q}$ is a parabola is treated in the same way and the degenerate case where $\widehat{Q}$ is a pair of lines cannot happen since $y_{1},\dots,y_{n}$ form the sides of a nondegenerate convex polygon. Hence, $\widehat{Q}$ is an ellipse and the claim follows.
\end{proof}

Projective rigidity of families $\mathcal{R}(1)$ and family $\mathcal{R}(2)$ has already been proved in Theorem~3.6 and Corollary~3.7 of \cite{Cu}. Here, we prove it as a consequence of Proposition~\ref{prop:chareg}.

\begin{cor}\label{cor:ProRig}
If two simplicial arrangements belong to the same combinatorial equivalence class belonging to family $\mathcal{R}(1)$ or family $\mathcal{R}(2)$, then they are conjugated by a projective transformation of $\mathbb{RP}^{2}$.
\end{cor}

\begin{proof}
For arrangements $\mathcal{A}(2k,1)$ with $k \geq 5$, we just have to apply Proposition~\ref{prop:chareg} to the star of the central vertex of the arrangement. Arrangements of combinatorial class $\mathcal{A}(6,1)$ are formed by the six lines containing two points in a configuration of four points $A,B,C,O$ such that $A,B,C$ form a triangle and $O$ belongs to the interior of the triangle. Up to projective transformation, there is only one such configuration so the arrangements of class $\mathcal{A}(6,1)$ are projectively rigid.
\par
Similarly, arrangements of $\mathcal{A}(8,1)$ are combinatorially equivalent to an arrangement formed by the sides of a square $ABCD$, their diagonals $AC,BD$ intersecting in $O$ and the two lines $OE$ and $OF$ where $E=AB \cap CD$ and $F=BC \cap AD$. If we forget the two lines $OE$ and $OF$, we obtain an arrangement of $\mathcal{A}(6,1)$ so any arrangement of $\mathcal{A}(8,1)$ is obtained from a projectively rigid configuration of four points exactly like the arrangements of  class $\mathcal{A}(6,1)$.
\par
Any arrangement of family $\mathcal{R}(2)$ is obtained from an arrangement of family $\mathcal{R}(1)$ without any choice by adding the line at infinity (the line containing the intersection points of the parallel sides of a regular polygon). Therefore, arrangements of families $\mathcal{R}(1)$ and $\mathcal{R}(2)$ are projectively rigid.
\end{proof}

\begin{rem}
In contrast with families $\mathcal{R}(1)$ and $\mathcal{R}(2)$, the moduli space of arrangements of $\mathcal{R}(0)$ formed by $n$ lines is of real dimension $n-3$. In Conjecture 2.17 of \cite{Gr0}, Gr\"{u}nbaum
asserts that all simplicial arrangement of $\mathbb{RP}^{2}$ that do not belong to $\mathcal{R}(0)$ are projectively rigid.
\end{rem}

\nopagebreak
\vskip.5cm
\end{document}